\documentclass[12pt]{amsart}

\usepackage{amsfonts,amssymb,graphicx,bm,enumerate}
\usepackage[small,heads=LaTeX,PS,nohug]{diagrams}

\newcommand{\excise}[1]{}

\numberwithin{equation}{section}
\pushQED{\qed}

\newtheorem{thm}{Theorem}[section]
\newtheorem{lemma}[thm]{Lemma}
\newtheorem{prop}[thm]{Proposition}
\newtheorem{cor}[thm]{Corollary}

\theoremstyle{definition}
\newtheorem{defn}[thm]{Definition}
\newtheorem{remark}[thm]{Remark}
\newtheorem{example}[thm]{Example}

\newcommand\BB{\mathbb{B}}
\newcommand\CC{\mathbb{C}}
\newcommand\EE{\mathbb{E}}
\newcommand\GG{\mathbb{G}}
\newcommand\II{\mathcal{I}}

\newcommand\NN{\mathbb{N}}
\newcommand\OO{\mathcal{O}}
\newcommand\PP{\mathbb{P}}
\newcommand\QQ{\mathbb{Q}}
\newcommand\RR{\mathbf{R}}
\newcommand\TT{\mathbb{T}}
\newcommand\UU{\mathbb{U}}
\newcommand\WW{\mathcal{W}}
\newcommand\XX{\mathcal{X}}
\newcommand\YY{\mathcal{Y}}
\newcommand\ZZ{\mathbb{Z}}

\newcommand\cE{\mathcal{E}}
\newcommand\cF{\mathcal{F}}
\newcommand\cG{\mathcal{G}}
\newcommand\cL{\mathcal{L}}
\newcommand\cM{\mathcal{M}}

\newcommand\fg{\mathfrak{g}}

\newcommand\PPJ{{\mathbb{P}_{\!J}}}
\newcommand\bbb{\mathfrak{b}}
\newcommand\del{\partial}
\newcommand\too{\longrightarrow}
\newcommand\ttt{\mathfrak{t}}
\newcommand\isom{\cong}
\newcommand\ldot{\,.\,}
\newcommand\minus{\smallsetminus}
\newcommand\goesto{\rightsquigarrow}

\newcommand\mapstoo{\longmapsto}

\newcommand\<{\langle}
\renewcommand\>{\rangle}

\renewcommand\th{\mathrm{th}}
\renewcommand\phi{\varphi}
\renewcommand\hat{\widehat}
\renewcommand\iff{\Leftrightarrow}

\newcommand\ol[1]{{\overline{#1}}{}}
\newcommand\wt[1]{{\widetilde{#1}}{}}

\DeclareMathOperator\Hom{Hom}
\DeclareMathOperator\Tor{\mathcal{T}\hspace{-.8ex}\mathit{or}}
\DeclareMathOperator\shfHom{\mathcal{H}\hspace{-.2ex}\mathit{om}}

\DeclareMathOperator\codim{codim}

\begin{document}

\title[Positivity and Kleiman transversality in equivariant $K$-theory]{Positivity and
       Kleiman transversality in equivariant $\bm{K}$-theory of homogeneous spaces}
\author{Dave Anderson}
\address{Department of Mathematics\\University of Michigan\\Ann Arbor, MI 48109}
\email{dandersn@umich.edu}
\author{Stephen Griffeth}
\address{School of Mathematics\\University of Minnesota\\Minneapolis, MN 55455}
\email{griffeth@math.umn.edu}
\author{Ezra Miller}
\address{School of Mathematics\\University of Minnesota\\Minneapolis, MN 55455}
\email{ezra@math.umn.edu}

\begin{abstract}
We prove the conjectures of Graham--Kumar \cite{GrKu} and
Griffeth--Ram \cite{GrRa} concerning the alternation of signs in the
structure constants for torus-equivariant $K$-theory of generalized
flag varieties $G/P$.  These results are immediate consequences of an
equivariant homological Kleiman transversality principle for the Borel
mixing spaces of homogeneous spaces, and their subvarieties, under a
natural group action with finitely many orbits.  The computation of
the coefficients in the expansion of the equivariant $K$-class of a
subvariety in terms of Schubert classes is reduced to an Euler
characteristic using the homological transversality theorem for
non-transitive group actions due to S.~Sierra.  A vanishing theorem,
when the subvariety has rational singularities, shows that the Euler
characteristic is a sum of at most one term---the top one---with a
well-defined sign.  The vanishing is proved by suitably modifying a
geometric argument due to M.~Brion in ordinary $K$-theory that brings
Kawamata--Viehweg vanishing to bear.
\end{abstract}


\date{10 March 2017}

\maketitle

\section{Introduction}\label{s:intro}

The structure constants of cohomology rings of homogeneous spaces tend
to exhibit positivity properties.  Combinatorics often enters, through
attempts to interpret positive quantities as enumerators, but it is by
geometric means that the positivity is often first---or most
easily---verified (a notable exception being the order of events
relating \cite{buchLR} and \cite{brion}).  In the typical setup, going back to Ehresmann
\cite{Ehr}, the cohomology ring in question possesses an additive
basis of classes carried by algebraic subvarieties.  Using the
transitive group action, as pioneered by Kleiman \cite{Kle}, these
Schubert subvarieties can be translated generically; subsequently
intersecting them yields cycles whose multiplicities are positive by
virtue of being~algebraic.  These multiplicities are the structure
constants, which are hence~positive.

That positivity extends beyond ordinary cohomology has recently been
demonstrated in two instances.  Graham generalized it to
torus-equivariant cohomology of the homogeneous spaces $G/P$
\cite{graham}, confirming and extending conjectures of Billey and
Peterson (cf.\ \cite[Section~4]{graham}).  At about the same time,
Brion proved it for ordinary $K$-theory \cite{brion}, after it had
been conjectured by Buch \cite{buchLR}.  The very notion of positivity
depends on the context, of course.  In $K$-theory, positivity means
sign alternation: if the dimension of a subvariety differs from the
expected dimension by~$i$ in a given intersection product, then the
sign of the coefficient on its class is~$(-1)^i$.  For equivariant
cohomology, the coefficients are polynomials; positivity means that,
expressed as polynomials in the simple roots, the coefficients are
nonnegative integers.

\raggedbottom%
In view of the developments for ordinary $K$-theory and equivariant
cohomology, it should come as no surprise that conjectures for
equivariant $K$-theory, posed by Griffeth--Ram \cite{GrRa} and
Graham--Kumar \cite{GrKu}, predict sign alternation in equivariant
$K$-theory.  These conjectures again make precise the notion of
positivity for polynomials in terms of which the alternation is
phrased.  The aim of this paper is to derive these conjectures
(Corollaries~\ref{c:GrKu}, \ref{c:GrKu'}, and~\ref{c:GrRa}) from an
appropriate generalization of Kleiman transversality
(Theorem~\ref{t:strong}).

Several special cases of our main results have been proved over the
last two decades.  The first steps in this direction concerned
multiplication by the class of a line bundle: positive ``Pieri''
formulas in equivariant $K$-theory were given by Fulton--Lascoux in
type $A$ \cite{fulton-lascoux}, and by Pittie--Ram \cite{pittie-ram}
and Mathieu \cite{Mat} for general $G/B$.  Griffeth--Ram verified
their conjecture for all rank $2$ groups, and Graham--Kumar proved
their conjecture for projective spaces, as well as Schubert expansions
of opposite Schubert varieties in any $G/B$.

Ordinary Kleiman transversality concerns the movement of subvarieties
of a homogeneous space $G/P$ into general position using the
transitive group action.  This has positivity consequences for
non-equivariant cohomology theories, including ordinary $K$-theory
of~$G/P$, because translation preserves rational equivalence.
Equivariantly, on the other hand, translation alters the classes of
cycles.  Our equivariant homological Kleiman transversality principle,
Theorem~\ref{t:kleiman}, and its strong version for subvarieties with
rational singularities, Theorem~\ref{t:strong}, therefore take place
on a Borel mixing space of~$G/P$---or, more precisely, a
finite-dimensional approximation~$\XX$---whose ordinary (i.e.,
non-equivariant) homological invariants are the equivariant invariants
of~$G/P$.

Roughly speaking, Theorems~\ref{t:strong} and~\ref{t:kleiman} say that
$\XX$ has a large group action, large enough so that the general
translate of the mixing space~$\YY$ of a torus-stable subvariety~$Y$
(or arbitrary coherent sheaf on~$G/P$) is homologically transverse
in~$\XX$ to the mixing spaces~$\XX^w$ of the opposite Schubert
subvarieties $X^w \subseteq G/P$.  Consequently, each coefficient in
the Schubert basis expansion of the torus-equivariant $K$-class
$[\OO_Y]$ can be computed as the Euler characteristic of a certain
``boundary'' divisor on the intersection $\YY \cap \gamma.\XX^w_J$
of the mixing space~$\YY$ with a general translate of another mixing
space~$\XX^w_J$ (Theorem~\ref{t:kleiman}).  To be more precise, mixing
spaces are bundles over a product~$\PP$ of large projective spaces
(Section~\ref{s:flag}), and $\XX^w_J$ is the restriction to a certain
product of projective subspaces, indexed by~$J$, of the mixing space
$\XX^w$.  The strong version in Theorem~\ref{t:strong} says that when
$Y$ has rational singularities, the Euler characteristic is an
alternating sum of terms in which only the last can be nonzero.

The action of the \emph{mixing group}~$\Gamma$, introduced in
Section~\ref{s:group}, is derived from the structure of the mixing
space~$\XX$ as a bundle with fiber $G/P$ over~$\PP$.  The base~$\PP$
has a transitive automorphism group.  The fibers $G/P$ obviously also
have transitive automorphism groups, but a~priori these only guarantee
automorphisms of~$\XX$ over open subsets of~$\PP$.  Constructing
global automorphisms depends on constructing sections of a group
scheme related to~$\XX$.  This, in turn, ultimately relies on a
certain positivity hypothesis on the torus action
(Section~\ref{sub:full}) that pervades all of our main results.

In view of the applications in Section~\ref{s:applications}, the
statements of our main results, particularly Theorems~\ref{t:main}
and~\ref{t:strong}, as derived from Theorems~\ref{t:kleiman}
and~\ref{t:vanishing}, contain always two flavors: one expands
equivariant classes in terms of equivariant Schubert structure sheaves
$\OO_w = [\OO_{X_w}]$, while the other expands in terms of Schubert
interiors $\xi_w = [\OO_{X_w}(-\del X_w)]$.  The boundary divisor
$\del X_w$ is the union of the Schubert varieties properly contained
in~$X_w$.  For positivity in the latter case, it does not suffice to
start with the structure sheaf~$\OO_Y$ of a subvariety with rational
singularities; only a twist $\OO_Y(-\del Y)$ by the ideal sheaf of an
effective boundary divisor supporting an ample line bundle will do.
The two flavors have nearly identical proofs: the nuanced differences
in the statements result from a symmetry between the opposite Schubert
variety $X^w$ and and arbitrary subvariety~$Y$ with rational
singularities; see the proof of Theorem~\ref{t:vanishing}.

The outline of our method comes from a combination of Anderson's proof
\cite{and} of Graham's equivariant cohomological positivity
\cite{graham} and Brion's proof of sign alternation for ordinary
$K$-theory \cite{brion}.  First, we translate equivariant statements
on $G/P$ into non-equivariant ones on the finite-dimensional
approximations of Borel mixing spaces in Section~\ref{sub:borel} and
Section~\ref{s:approx}.  After stating our main results and their
previously conjectured corollaries in Sections~\ref{s:main}
and~\ref{s:applications}, we construct the ``sufficiently transitive''
group action on the mixing space in Section~\ref{s:group}.  This
results in the weak Kleiman transversality principle in
Section~\ref{s:transverse}.  The difference between the weak version
and its strong one is the vanishing result in
Section~\ref{s:vanishing}, particularly Theorem~\ref{t:vanishing}.
The proof requires a result on lifting rational singularities under
smooth morphisms in Section~\ref{s:rat-sings}, along with explicit
constructions of such smooth morphisms, based on Bott--Samelson
resolutions of singularities, in Section~\ref{s:bs}.

What makes things simpler in cohomology, as opposed to $K$-theory, is
that cohomology only requires knowledge on a Zariski dense open
subset.  Each of the relevant cohomology computations \cite{and} is
carried by an intersection that occurs in one cell of a paving of the
mixing space by bundles of affine spaces.  As the group action is
transitive on each such bundle, ordinary Kleiman transversality
suffices.  One must then push down to the base of the mixing space,
but this operation transfers the positivity to the resulting class.

What fails in $K$-theory?  First, unable to restrict to an open cell,
we must instead attend to coherent sheaves on closed subvarieties,
where the group action is not transitive.  Second, pushing forward to
the base can have higher direct images, a~priori causing mixed
negative and positive coefficients.

Getting around the second obstacle is simple, in principle: impose
vanishing of the higher direct images.  In practice, this is
accomplished by stipulating rational singularities, which the
Graham--Kumar conjecture \cite[Conjecture~7.1]{GrKu} does explicitly,
taking the cue from Brion's phrasing of the result in ordinary
$K$-theory \cite[Theorem~1]{brion}.  We proceed by suitably modifying
Brion's geometric argument that brings Kawamata--Viehweg vanishing to
bear.

Dealing with the obstacle of non-transitivity is harder in principle,
but it has been made simple in practice by the happy circumstance of
recent developments.  After Kleiman transversality was generalized to
non-transitive group actions in cohomology by Speiser \cite{speiser},
it was recently generalized to homological transversality in
$K$-theory for transitive actions by Miller and Speyer
\cite{torVanish}.  More recently still, Sierra formulated and proved a
$K$-theoretic version for non-transitive group actions \cite{sierra},
and this (Theorem~\ref{t:sierra}) is the version we use in the proof
of (weak) equivariant homological Kleiman transversality,
Theorem~\ref{t:kleiman}.

Looking to the future, it seems the next step would be to combine the
equivariant positivity statements here with Mihalcea's in the
equivariant quantum setting \cite{mihalcea}.  In our (non-quantum)
setup we can use functoriality to reduce to the full (generalized)
flag variety $G/B$; this is one of many things that will have to
change to deal with the quantum situation.

\subsection*{Acknowledgements}

The authors are grateful to Michel Brion for his work on $K$-theory of
homogeneous spaces \cite{brion}, which has been indispensable in our
investigations; for personally providing additional insight into the
subtleties of that work; and for sending helpful comments on a draft
of this article.  Hearty thanks go to Shrawan Kumar for numerous
detailed comments and significant corrections on an earlier draft.  We
also wish to thank Seth Baldwin and Hiroshi Naruse for valuable corrections, and Susan
Sierra for discussions regarding $K$-theory and Cohen--Macaulayness.
Substantial parts of this work were completed during two visits by DA
to the University of Minnesota.  DA was partially funded by an RTG
graduate fellowship, NSF Grant 0502170.  SG and EM were partially
funded by NSF Career Grant DMS-0449102.

\section{Flag varieties, mixing spaces, and $K$-theory}\label{s:flag}%

Write $\NN \subset \ZZ \subset \CC$ for the monoid of nonnegative
integers, the ring of integers, and the field of complex numbers.  All
of our schemes are separated and of finite type over~$\CC$.  A variety
is assumed to be reduced and equidimensional, but not necessarily
irreducible.  If a group~$G$ acts on $Y$ on the right and on $Z$ on
the left, then $Y \times^G Z$ is defined to be the quotient of the
product $Y \times Z$ by the relation $(y\ldot g, z)\sim (y,g\ldot z)$.

\subsection{Lie theory}\label{sub:lie}\flushbottom

We refer to Borel \cite{Bor} for the following standard facts and
notation.  Let $G$ be a complex semisimple algebraic group of adjoint
type, and fix a choice $T \subseteq B \subseteq G$ of a maximal torus
and Borel subgroup.  These have Lie algebras $\ttt \subseteq \bbb
\subseteq \fg$.  The weight lattice $\Hom(T,\CC^*)$ of~$T$ contains
the set~$R$ of roots.  Write $R^+$ and $R^-$ for the positive and
negative roots, respectively (so $R^+$ is the set of nonzero weights 
for the action of $T$ on $\bbb$).  Write $\Delta = \{\alpha_1,
\ldots, \alpha_n\} \subseteq R^+$ for the simple roots; thus every 
positive root $\alpha\in R^+$ can be written as $\alpha = \sum_i k_i 
\alpha_i$ with $k_i\in \NN$.

Since $G$ is adjoint, the root and weight lattices are the same, and
$\Delta$ is a basis for the weight lattice.  For a weight $\lambda$,
write $e^\lambda\colon T \to \CC^*$ for the corresponding character.

The normalizer $N(T)$ of the torus~$T$ in~$G$ has the Weyl group $W =
N(T)/T$ as its quotient.  (Following a common abuse of notation, we
sometimes identify $w\in W$ with a chosen representative in
$N(T)\subseteq G$; the choice will never matter.)  The simple roots
$\alpha_i$ determine simple reflections $s_i \in W$, and these
generate~$W$.  The length~$\ell(w)$ of an element $w \in W$ is the
smallest number $\ell$ such that $w$ has an expression $w = s_{i_1}
\cdots s_{i_\ell}$ as a product of simple reflections.  When $\ell =
\ell(w)$, such an expression is a reduced word for~$w$.  Let $w_\circ
\in W$ be the (unique) longest element.  The (strong) Bruhat partial order
on~$W$ is defined by setting $v \leq w$ if $v$ has a reduced word that
occurs as a subword (not necessarily consecutive) of a reduced word for~$w$.

\subsection{Flag varieties and Schubert varieties}\label{sub:flag}

The central object of this paper is the quotient $X = G/B$, known as
the (complete) flag variety of~$G$.  Let $B^- = w_\circ B w_\circ$ be
the opposite Borel subgroup, so $T = B \cap B^-$.  The flag variety is
paved by the Schubert cells $C_w = BwB/B \isom \CC^{\ell(w)}$ and also
by the opposite Schubert cells $C^w = B^-wB/B \isom \CC^{\dim X -
\ell(w)}$:
$$%
  X = \coprod_{w \in W} C_w = \coprod_{w \in W} C^w.
$$
The Schubert varieties $X_w$ and opposite Schubert varieties $X^w$ are
the closures in~$X$ of the cells~$C_w$ and~$C^w$, respectively.
Bruhat order encodes containments among them:
$$%
  X_v \subseteq X_w \ \iff\ v \leq w \ \iff\ X^v \supseteq X^w.
$$

More generally, if $P \subseteq G$ is a parabolic subgroup, the
partial flag variety $G/P$ corresponding to~$P$ has a cell
decomposition
$$%
  G/P = \coprod_{w \in W^P} C_w,
$$
where $C_w = BwP$ and $W^P$ is the set of minimal length
representatives for the cosets of~$W$ modulo its parabolic subgroup
corresponding to~$P$.  Again write $X_w$ and $X^w$ for the Schubert
and opposite Schubert varieties, the closures of~$C_w$ and~$C^w$
in~$G/P$.

The Schubert varieties $X_v$ and $X^w$ intersect properly and
generically transversally in the \emph{Richardson variety}~$X_v^w$.
In particular, $X_v^w$ is empty unless $v\geq w$, and the intersection
$X_w\cap X^w$ is transverse at the point~$wB$; moreover, Richardson 
varieties are irreducible.  Schubert varieties are
Cohen--Macaulay and have rational singularities \cite{ramanathan}, and
the same is true of Richardson varieties, using \cite[Lemmas~1
and~2]{brion}.

\subsection{Borel mixing spaces and approximations}\label{sub:borel}

We recall some basic notions concerning the Borel mixing space
construction; for details, see \cite{eq} or \cite{EdGr}.  Let $S$ be
an algebraic torus of dimension~$r$.  Fix a basis
$\{\beta_1,\ldots,\beta_r\}$ for the weight lattice of~$S$, and use the 
negatives of these weights to identify $S$ with~$(\CC^*)^r$.  
In our applications, $S$ will be a
subtorus of the maximal torus $T \subseteq G$ fixed in
Section~\ref{sub:lie}.

The universal principal $S$-bundle $\EE S \to \BB S$ is the union of
finite-dimensional algebraic approximations $\EE_m S \to \BB_m S$,
which may be constructed as
$$%
  \EE_m S = (\CC^{m+1} \minus 0)^{\times r} \to (\PP^m)^{\times r} =
  \BB_m S,
$$
where $S \isom (\CC^*)^r$ acts on $(\CC^{m+1} \minus 0)^{\times r}$
diagonally by the standard action; 
that is, the action of $S$ on the $i$th 
factor of $(\CC^{m+1} \minus 0)^{\times r}$ is by $-\beta_i$.  
We write
$$%
  \PP = \EE_m S / S = \BB_m S
$$
for some fixed sufficiently large $m \gg 0$.

If $Y$ is a scheme with a left $S$-action, the \emph{Borel mixing
space} is $\EE S \times^S Y$.  As with the universal principal
$S$-bundle, we will only use algebraic approximations $\YY = \EE_m S
\times^S Y$, for some fixed $m \gg 0$.  Thus $\YY$ is a
Zariski-locally trivial fiber bundle over~$\PP$ with fiber~$Y$.  We
view the transition $Y \goesto \YY$ from an $S$-scheme to its
approximate mixing space as a functor on $S$-schemes, and we always
indicate it by changing from roman to calligraphic font.  When $X =
G/P$, we denote by~$p$ the projection $\XX \to \PP$.

A section $\PP \to \YY$ is equivalent to an $S$-equivariant map $\EE_m
S \to Y$, as one sees from the following fiber diagram, where the
horizontal maps are principal $S$-bundles.
\begin{diagram}
  \EE_m S \times Y &     \rTo    & \YY \\
         \dTo      &\boxtimes\ \ & \dTo\\
        \EE_m S    &     \rTo    & \PP
\end{diagram}

\subsection{Positivity of subtori}\label{sub:full}

The basis $\{\beta_1,\ldots,\beta_r\}$ in Section~\ref{sub:borel} for
the weight lattice of the subtorus $S \subseteq T$ is
\begin{itemize}
\item
\emph{positive} if the restrictions $\alpha_1|_S, \ldots, \alpha_n|_S$
of the simple roots $\alpha_1,\ldots,\alpha_n$ are all nonnegative
integer combinations of $\beta_1,\ldots,\beta_r$; and
\item
\emph{full} if it is positive, and each $\beta_i$ equals the
restriction of some simple root.
\end{itemize}

The positivity hypothesis will arise systematically, as it is
essential to the geometry in our proof of Theorem~\ref{t:strong}.
Notably, it guarantees that the mixing group $\Gamma$ is big enough:
positivity begets sections.  On the other hand, fullness arises as an
essential hypothesis only once in this paper: we mention it in
Corollary~\ref{c:GrKu}, with regard to the diagonal subtorus $S
\subseteq T \times T$ inside $G \times G$, so that it can be applied
in Corollary~\ref{c:GrKu'}.

\subsection{Restrictions and boundary divisors}\label{sub:restriction}

For each $j \in \{0,\ldots,m\}$, fix a subspace $\PP^j \subseteq
\PP^m$.  Then, for any $r$-tuple $J = (j_1,\ldots,j_r)$ of integers
with $0 \leq j_i \leq m$, write
$$%
  \PPJ = \PP^{j_1} \times \cdots \times \PP^{j_r} \subseteq \PP
$$
and similarly
$$%
  \PP^J = \PP^{m-j_1} \times \cdots \times \PP^{m-j_r} \subseteq \PP.
$$
Set $|J| = j_1+\cdots+j_r = \dim \PPJ = \codim \PP^J$.  The subvariety
$\PPJ$ has \emph{boundary divisor}
$$%
  \del\PPJ = \PP_{(j_1-1,j_2,\ldots,j_n)} \cup
  \PP_{(j_1,j_2-1,\ldots,j_n)} \cup \cdots \cup
  \PP_{(j_1,j_2,\ldots,j_n-1)}.
$$

Our reason for defining $\PPJ$ is that, for $S$-invariant subschemes
$Y \subseteq X = G/B$, we will need to consider the restrictions
$$%
  \YY_J = p^{-1}(\PPJ) \cap \YY \subseteq \XX \quad \text{and} \quad
  \YY^J = p^{-1}(\PP^J) \cap \YY \subseteq \XX
$$
of the bundles $\YY$ to~$\PPJ$ and $\PP^J$.  In particular, when $Y =
X^w$ is an opposite Schubert variety, so $\YY_J = \XX^w_J =
\XX^w|_{\PPJ}$, we will additionally need to consider the
\mbox{\emph{boundary divisor}}
$$%
  \del\XX^w_J = \big(\XX^w|_{\del\PPJ}\big) \cup \big(\bigcup_{v>w}
  \XX^v_J\big).
$$
This variety is Cohen--Macaulay, by \cite[Lemma~4]{brion}.

We will on many occasions need sheaves of the form $\OO_Y(-\del Y)$
for which a Weil divisor $\del Y$ has been defined.  When $\del Y$ is
effective, $\OO_Y(-\del Y) = \II(\del Y)$ is the ideal sheaf of~$\del
Y$ in~$\OO_Y$.  Once $\del Y$ has been defined, we write $\del\YY$ for
the corresponding Weil divisor on the mixing space.  In what follows,
we will often write $\OO_Y(-\del)$ instead of $\OO_Y(-\del Y)$ because
our boundary divisors can be notationally complicated varieties.  When
$\OO_Y(-\del Z)$ is written, it serves to emphasize that $Z \neq Y$.

\subsection{Line bundles and canonical sheaves}\label{sub:bundles}

A character $\lambda\colon  S \to \CC^*$ defines a geometric line bundle
$$%
  \OO(\lambda) = \EE_m S \times^S \CC_{\lambda}
$$
on~$\PP$, where $\CC_{\lambda}$ is the one-dimensional representation
in which $z \ldot v = \lambda(z) v$ for \mbox{$v \in
\CC_\lambda$} and $z \in S$.  Using the basis $\{\beta_1,\ldots,\beta_r\}$ 
to identify $S\isom(\CC^*)^r$ as in \S\ref{sub:borel}, we have
$$%
  \OO(\beta_i) \isom p_i^*\OO(1),
$$
where $p_i$ is the projection on the $i$th factor of $\PP = (\PP^m)^{\times r}$.  
(This explains why we use the negative weights to define the isomorphism $S\isom(\CC^*)^r$.)  
Positivity for line bundles defines a partial order on
the weight lattice of~$S$, in which $\lambda \geq 0$ if and only if
$\OO(\lambda)$ possesses nonzero global sections; equivalently,
\mbox{$\lambda = c_1\beta_1 + \cdots + c_r\beta_r \geq 0$} if and only
if $c_i \geq 0$ for all~$i$.

Similarly, $\lambda$ also defines a line 
bundle
$$%
  \cL_\lambda = G \times^B \CC_{-\lambda},
$$
on $X=G/B$, by extending the character to $B$.  The line bundle
$\cL_\lambda$ is to be distinguished from the equivariantly nontrivial
but non-equivariantly trivial line bundle $e^\lambda = X \times
\CC_\lambda$.

\begin{example}\label{ex:rho}
When the character is $2\rho = \sum_{\alpha \in R^+} \alpha$, the line
bundle $\cL_{2\rho}$ is very ample.  By considering the
simply-connected form of~$G$, we find that $\cL_{2\rho}$ has a square
root~$\cL_\rho$, which is also a $B$-equivariant very ample line
bundle.  In general, this line bundle is not equivariant for the
adjoint torus, but the bundle $e^\rho\cL_\rho$ is.
\end{example}

For $X=G/B$, we have $\omega_X \isom \cL_{-2\rho}$,
$\omega_{X_w} \isom e^{-\rho}\cL_{-\rho}\otimes\OO_{X_w}(-\del)$, and 
$\omega_{X^w} \isom e^{\rho}\cL_{-\rho}\otimes\OO_{X^w}(-\del)$ as
equivariant sheaves \cite[Proposition~2.2(a-b)]{GrKu}.  Similarly,
for the Bott--Samelson varieties $\wt{X}_w \xrightarrow{\phi} X_w$, we
have $\omega_{\wt{X}_w} \isom \phi^*e^{-\rho}\cL_{-\rho}\otimes
\OO_{\wt{X}_w}(-\del)$ using \cite[Proposition~2]{ramanathan}.

\subsection{Equivariant $K$-theory}\label{sub:Ktheory}

For this subsection, let $X$ be any smooth variety with a left action
of the torus~$S$.  (Shortly, we will return to $X = G/P$ and $S
\subseteq T$, a torus in~$G$.)  Denote by $K_S(X)$ the Grothendieck
ring of $S$-equivariant vector bundles on~$X$.  The representation
ring equals the group algebra
$$%
   R(S) = \ZZ[\Lambda]
        = \bigoplus_{\lambda \in \Lambda} \ZZ \cdot e^\lambda
$$
of the weight lattice $\Lambda = \Hom(S,\CC^*)$ of~$S$.  It coincides
with the equivariant Grothen\-dieck ring of a point.  Writing $\pi$
for the projection to a point, the pullback $\pi^*$ therefore makes
$K_S(X)$ into an $R(S)$-module.

Since $X$ is smooth, the natural $R(S)$-module homomorphism $K_S(X)
\to K^S(X)$ to the Grothendieck group of $S$-equivariant coherent
sheaves on~$X$ is an isomorphism.  The product of the classes of two
coherent sheaves $\cE$ and~$\cF$ is the alternating sum
$$%
  [\cE] \cdot [\cF] = \sum_{i \geq 0} (-1)^i [\Tor_i^X(\cE,\cF)]
$$
of their Tor sheaves.

The $K$-homology group $K^S$ pushes forward along proper morphisms: $X
\xrightarrow{q} Y$ yields
$$%
  q_*[\cF] = \sum_{i \geq 0} (-1)^i [R^i q_*\cF],
$$
the point being that all higher direct images are coherent.  In
particular, if $X$ is smooth and proper, there is an $R(S)$-bilinear
pairing on $K_S(X)$ given by
$$%
  \<[\cE],[\cF]\>_S = \pi_*([\cE]\cdot[\cF]),
$$
where $\pi$ is the projection to a point.

\subsection{Equivariant $K$-theory of flag varieties}\label{sub:Kflag}

Resume the case $X = G/P$ acted on by a torus $S \subseteq T$.  Since
the subvarieties $X_w$ and $X^w$ are $S$-stable, their structure
sheaves are quotients of $\OO_X$ by $S$-stable ideal sheaves and hence
$S$-equivariant.  Let
$$%
  \OO_w = [\OO_{X_w}] \quad \text{and} \quad \OO^w = [\OO_{X^w}]
$$
be the classes of the structure sheaves of the Schubert varieties and
opposite Schubert varieties in~$K_S(X)$.  Because of the cell
decompositions in Section~\ref{sub:flag}, the sets $\{\OO_w\}_{w \in
W^P}$ and $\{\OO^w\}_{w \in W^P}$ indexed by the minimal length coset
representatives are bases for~$K_S(X)$ as an $R(S)$-module.  Let
$\xi^w = \big[\OO_{X^w}(-\del)\big]$, where $\del = \del X^w =
\bigcup_{v>w} X^v$ is the boundary of~$X^w$, and $\xi_w =
\big[\OO_{X_w}(-\del)\big]$, where $\del = \del X_w = \bigcup_{v<w}
X_v$.  Then $\{\xi^w\}_{w\in W^P}$ and $\{\xi_w\}_{w\in W^P}$ are two
more bases for~$K_S(X)$.

\begin{lemma}[{\cite[Proposition~2.1]{GrKu}}]\label{l:duality}
The $\OO$ and $\xi$ bases of~$K_S(G/P)$ are dual:
$$%
  \<\OO_w, \xi^v\>_S = \delta_{w,v} \in R(S)
  \quad\text{and}\quad
  \<\OO^w, \xi_v\>_S = \delta_{w,v} \in R(S).
$$
\end{lemma}

Further basic information and notation concerning the equivariant
$K$-theory of flag varieties must wait until Section~\ref{s:approx},
where it shown how the ordinary $K$-theory of mixing spaces
approximates it.

\subsection{Homological transversality}\label{sub:transversality}

Our results depend on a certain kind of transversality that simplifies
the $K$-theoretic product of two coherent sheaves.  This
simplification arises separately a couple of times, in the proof of
Conjecture~\ref{c:GrRa}, and in Section~\ref{s:transverse} as part of
the proof of our main result, Theorem~\ref{t:main}.

Two quasicoherent sheaves $\cE$ and $\cF$ on an arbitrary variety~$X$
are \emph{homologically transverse} if all of their higher Tor sheaves
vanish:
$$%
  \Tor_j^X(\cE,\cF) = 0 \quad \text{for all } j \geq 1.
$$
If $\cE = \OO_Y$ is the structure sheaf of a subvariety $Y \subseteq
X$, we say that $\cF$ is homologically transverse to~$Y$.  If $X$ is
complete and nonsingular, and $Y,Z \subseteq X$ are homologically
transverse subvarieties, then in $K(X)$,
$$%
  [\OO_Y] \cdot [\OO_Z] = [\OO_Y \otimes \OO_Z] = [\OO_{Y \cap Z}],
$$
where $Y \cap Z$ is the scheme-theoretic intersection.  When $Y$ and
$Z$ intersect properly (i.e., the sum of their codimensions equals the
codimension of every component of their intersection) and both are
Cohen--Macaulay, then they are homologically transverse, and their
intersection is Cohen--Macaulay; this is the content of
\cite[Lemma~1]{brion}.

We shall need the following special case of a theorem due to Sierra
\cite{sierra}.

\begin{thm}\label{t:sierra}
Let $X$ be a variety with a left action of an algebraic group $G$ and
let $\mathcal{F}$ be a coherent sheaf on $X$.  Suppose that
$\mathcal{F}$ is homologically transverse to the closures of the
$G$-orbits on $X$.  Then for each coherent sheaf $\mathcal{E}$ on $X$
there is a Zariski-dense open set $U \subseteq G$ such that
$\Tor_j^X(\cE, g\ldot\cF) = 0$ for all $j \geq 1$ and all $g \in U$.
\end{thm}

\subsection{Relative Kawamata--Viehweg vanishing}\label{sub:kv}

Our proof of Theorem~\ref{t:vanishing} relies on a relative form of
Kawamata--Viehweg vanishing, the relevant version of which we extract
from \cite[Corollary~6.11]{EsVi}.  In the statement, $f$-nef means
\emph{$f$-numerically effective}: the line bundle~$\cM$ on $\wt Z$ has
nonnegative intersection with every curve contained in a fiber of~$f$.
In addition, \emph{$f$-big} means that the powers of~$\cM$ give rise
to projective morphisms that preserve the dimension of every general
fiber of~$f$.

\begin{thm}\label{t:kv}
Let $f\colon  \wt Z \to Z$ be a proper surjective morphism of varieties,
with $Z$ nonsingular.  Let $\cM$ be a line bundle on~$\wt Z$ such that
$\cM^N(-D)$ is $f$-nef and $f$-big for a normal crossing divisor $D =
\sum_{j=1}^r a_j D_j$, where $0 < a_j < N$ for all~$j$.  Then
$$%
  R^i f_*(\cM \otimes \omega_\wt Z) = 0 \text{ for all } i > 0.
$$
\end{thm}

\section{Approximating equivariant $K$-theory}\label{s:approx}

Resume the notation from Section~\ref{s:flag}, including $X = G/P$ and
an $r$-dimensional torus \mbox{$S \subseteq T$} in~$G$, along with a
(not necessarily positive) basis $\beta_1,\ldots,\beta_r$ for the
weight lattice of~$S$.  What justifies our omission of the integer~$m$
from the notation for approximate Borel mixing spaces in
Section~\ref{sub:borel}?  The essential idea is that, in analogy with
the observation by Totaro, Edidin, and Graham \cite{totaro,EdGr}
underlying equivariant Chow theory, the $K$-theory of the approximate
mixing spaces has a well-behaved limit as $m$ increases without bound.
This analogy has been precisely formulated by Edidin and Graham
themselves in their work on equivariant Riemann--Roch \cite{EdGrK}.
For us, the required consequence is as follows.

\begin{prop}\label{p:approx}
Let $X = G/P$ and $S \subseteq T$ a torus in~$G$.  An equation holds
in~$K_S(X)$ if and only if its image holds in $K(\XX) = K_S(\EE_m S
\times X)$
for some large~$m$.
\end{prop}
\begin{proof}
Let $\hat R(S)_\QQ$ denote the completion of the rational
representation ring $R(S) \otimes_\ZZ \QQ$ at its augmentation ideal.
Concretely, the augmentation ideal of $R(S) \cong \ZZ[\Lambda]$ is
generated by the elements $1 - e^\lambda$ for all $\lambda$ in the
weight lattice~$\Lambda$.

The natural morphism $R(S) \to \hat R(S)_\QQ$ is clearly injective.
Tensoring this morphism with $K_S(X)$ yields the natural map from
$K_S(X)$ to its completion at the augmentation ideal of~$R(S)$ because
$K_S(X)$ is finitely generated as an $R(S)$-module.  Moreover, the
morphism remains injective upon this tensoring because $K_S(X)$ is
flat (in fact, free) as an $R(S)$-module by Lemma~\ref{l:duality}.

Next, observe that our system $\CC^{m \times r}$ of
$S$-representations and open subsets $\EE_m S$ constitute a ``good
system of representations'' in the sense of \cite[Section~2.1]{EdGrK}.
This ``goodness'' is easy to verify: it amounts essentially to
checking that $S$ acts freely on the open sets, the system is closed
under direct sum, and the complements of the open sets are linear
subspaces; the details are omitted.

In the presence of goodness \cite[Theorem~2.1]{EdGrK} says the
topology on $K_S(X)$ coincides with the one induced by the kernels of
the surjections $K_S(\CC^{m \times r} \times X) \to K_S(\EE_m S \times
X)$ induced by pullback.  The desired result therefore follows from
injectivity of the homomorphism to the completion.
\end{proof}

Use bars to distinguish classes in the ordinary $K$-ring $K(\XX)$ of
the mixing space from their preimages in $K_S(X)$.  Thus, we write
$\ol\OO_w = [\OO_{\XX_w}]$ and $\ol\OO^w = [\OO_{\XX^w}]$ for the
usual and opposite Schubert classes, as well as $\ol\xi^w$
and~$\ol\xi_w$ for their duals (Section~\ref{sub:Kflag}).

\begin{prop}\label{p:basis}
$K(\XX)$ is a $K(\PP)$-algebra with additive $K(\PP)$-bases
$\big\{\ol\OO_w\big\}{}_{w \in W^P}$ and $\big\{\ol\OO^w\big\}{}_{w
\in W^P}$.  The dual $K(\PP)$-bases are $\big\{\ol\xi^w\big\}{}_{w \in
W^P}$ and $\big\{\ol\xi_w\big\}{}_{w \in W^P}$, respectively.
\end{prop}
\begin{proof}
The corresponding statement for $K_S(X)$ as an algebra over~$R(S)$ is
a consequence of Lemma~\ref{l:duality}.  The desired result follows
from the considerations in the proof of Proposition~\ref{p:approx}:
$K_S(\EE_m S \times X) = K(\XX)$ is the quotient of $K_S(\CC^{m \times
r} \times X) = K_S(X)$
modulo the kernel of the surjective homomorphism $R(S) \to K(\PP)$.
\end{proof}

\begin{cor}\label{c:basis}
The classes $[\OO_{\PP^J}]$ are a $\ZZ$-basis for~$K(\PP)$.  Set
\mbox{$\OO^J = p^*[\OO_{\PP^J}] \in K(\XX)$}.  The ordinary $K$-theory
$K(\XX)$ has additive $\ZZ$-bases
$$%
  \big\{\OO^J \cdot \ol\OO_w\big\}_{J,w}
  \ \text{ and }\ 
  \big\{\OO^J \cdot \ol\xi_w\big\}_{J,w},
  \quad\text{where}\quad
  J \in \{0,\ldots,m\}^r \text{ and } w \in W^P.
$$
Moreover, $\OO^J \cdot \ol\OO_w = [\OO_{\XX_w^J}]$ and $\OO_J \cdot
\ol\OO^w = [\OO_{\XX^w_J}]$, where \mbox{$\OO_J = p^*[\OO_{\PP_J}] \in
K(\XX)$}.\qed
\end{cor}

The importance of the $K$-classes $\OO^J$ and~$\OO_J$ on the mixing
space is that they provide a geometric interpretation for monomials in
the ``variables'' $1 - e^{-\lambda}$.

\begin{lemma}\label{l:monomials}
Let $J \in \{0,\ldots,m\}^r$.  The ordinary $K$-class $\OO^J \in
K(\XX)$ is the image of the equivariant ``monomial'' class
$(1-e^{-\beta_1})^{j_1} \cdots (1-e^{-\beta_r})^{j_r} \in K_S(X)$.
\end{lemma}
\begin{proof}
Use the exact sequence $0 \to \cL_{-\beta_i} \to \OO_\PP \to \OO_{H^i}
\to 0$, where $H^i = \PP^{\beta_i}$ is the component of the boundary
$\del\PP$ having $\PP^{m-1}$ in the $i^\th$ slot.  It immediately
implies that $\OO^{\beta_i} = 1 - e^{-\beta_i}$.  Clearly
$\OO^{d\beta_i} = (1 - e^{-\beta_i})^d$; now use the K\"unneth
formula.
\end{proof}

\begin{remark}
Viewing the Chow ring as the associated graded ring of $K$-theory, $1
- e^{-\lambda}$ gives rise to the class~$\lambda$ (the lowest degree
term in the expansion of $1 - e^{-\lambda}$ as a power series).  This
is another indication that $1 - e^{-\lambda}$ should be
considered~``positive''.
\end{remark}

Since $\XX$ is compact, its ordinary $K$-theory has a pairing given by
$\<\alpha,\beta\> = \chi(\alpha \cdot \beta)$, where $\chi\colon  K(\XX) \to
\ZZ$ is the Euler \mbox{characteristic}.

\begin{lemma}\label{l:bundle-duality}
Let $I,J \in \{0,\ldots,m\}^r$ and $v,w \in W^P$.  Using $(-\del)$ as
in Section~\ref{sub:restriction},
$$%
  \big\<\OO^J \cdot \ol\OO_w,\, \OO_I(-\del) \cdot
  \ol\OO^v(-\del)\big\> = \big\<\OO^J \cdot \ol\OO_w,\,
  [\OO_{\XX^v_I}(-\del)]\big\> = \delta_{(J,w),(I,v)}.
$$
\end{lemma}
\begin{proof}
Follow \cite[Prop.~2.1]{GrKu} and \cite[Lemma~1]{brion}; the details
are omitted.
\end{proof}

\section{Main theorems}\label{s:main}

\begin{thm}\label{t:main}
Let a torus $S \subseteq T$ with a positive basis
(Section~\ref{sub:full}) act on $X = G/P$.  Fix an $S$-stable
subvariety $Y \subseteq X$ and an $S$-stable, Cohen--Macaulay
effective divisor $\del \subset Y$ that supports an ample line bundle on~$Y$.
Let $\YY \subseteq \XX$ be the corresponding mixing spaces, which are
fiber bundles over\/~$\PP$.  Using the bases for $K(\XX)$ in
Corollary~\ref{c:basis}, define $c_{J,w}^{\,Y}$ and $d_{J,w}^{\,Y}$~by
$$%
  \big[\OO_\YY\big] =
  \sum_{J,w} c_{J,w}^{\,Y}\,\OO^J \cdot \ol\OO_w
  \quad\text{and}\quad
  \big[\OO_{\YY}(-\del)\big]=\sum_{J,w} d_{J,w}^{\,Y}\,\OO^J \cdot \ol\xi_w,
$$
the equations being in~$K(\XX)$.  If\/~$Y$ has rational singularities,
then
$$%
  (-1)^{\dim Y - \ell(w) + |J|} c_{J,w}^{\,Y}
  \quad\text{and}\quad
  (-1)^{\dim Y - \ell(w) + |J|} d_{J,w}^{\,Y}
$$ 
are nonnegative integers.
\end{thm}

Just as positivity for cohomology is an immediate consequence of
Kleiman transversality, Theorem~\ref{t:main} is an immediate
consequence of the following ``positive'' homological interpretation
of the coefficients $c_{J,w}^{\,Y}$ and $d_{J,w}^{\,Y}$ resulting from
a generic translation.

\begin{thm}[Strong equivariant homological Kleiman transversality]\label{t:strong}
Assume the situation of Theorem~\ref{t:main}.  There is an algebraic
\emph{mixing group}~$\Gamma$ acting on~$\XX$ with finitely many
orbits, the closure of each being a mixing space~$\XX_w$ of some
Schubert variety $X_w \subseteq X$.  Fix a general closed point
$\gamma \in \Gamma$, and write $\gamma\cF$ for the pushforward of any
sheaf~$\cF$ on~$\XX$ under multiplication by $\gamma \in \Gamma$.
If\/~$Y$ has rational singularities,~then
$$%
  (-1)^{\dim Y - \ell(w) + |J|} c_{J,w}^{\,Y} = \dim H^{\dim Y -
  \ell(w) + |J|}\big(\YY\,\cap\,\gamma.\XX^w_J, \OO_{\YY\,\cap\,
  \gamma.\XX^w_J}(-\del)\big),
$$
where the boundary divisor is $\del = \del(\YY\,\cap\,\gamma.\XX^w_J)
= \YY \cap \del(\gamma.\XX^w_J)$, and
$$%
  (-1)^{\dim Y - \ell(w) + |J|} d_{J,w}^{\,Y} = \dim H^{\dim Y -
  \ell(w) + |J|}\big(\YY_J\,\cap\,\gamma.\XX^w, \OO_{\YY_J\,\cap\,
  \gamma.\XX^w}(-\del_\gamma)\big),
$$
where the boundary divisor is $\del_\gamma = \big(\del\YY_J\big) \cap
\gamma.\XX^w$, with $\del\YY_J = \big(\YY|_{\del\PPJ}\big) \cup
\big(\del\YY\big)|_{\PPJ}$.
\end{thm}
\begin{proof}
The group $\Gamma$ is constructed in Section~\ref{s:group}, and the
statement about its orbits is Lemma~\ref{l:orbits}.  The construction
of~$\Gamma$ is where positivity of the basis for the weight lattice
of~$S$ is crucial, for it guarantees that a certain vector bundle
possesses enough sections.  Knowing the set of orbit closures allows
us easily to express the coefficients $c_{J,w}^{\,Y}$
and~$d_{J,w}^{\,Y}$ as Euler characteristics in
Theorem~\ref{t:kleiman}, using Sierra's homological transversality
(Theorem~\ref{t:sierra}) for group actions that are not necessarily
transitive.  The desired result follows from the more difficult
Theorem~\ref{t:vanishing}, which says that each Euler characteristic
is an alternating sum of terms in which only the last can be nonzero.
\end{proof}

Having already explained the roles of Sections~\ref{s:group},
\ref{s:transverse}, and~\ref{s:vanishing} in the proof of
Theorem~\ref{t:strong}, let us complete the discussion by explaining
the roles of Sections~\ref{s:rat-sings} and~\ref{s:bs}.  The proof of
the vanishing result in Theorem~\ref{t:vanishing} is a modification of
Brion's proof of the corresponding vanishing for ordinary $K$-theory
\cite{brion}, which is modeled on a Kleiman-type transversality
argument.  The main difficulty in extending Brion's methods to our
situation is the failure of transitivity for our group action
on~$\XX$.  It requires us to produce an intermediate result on lifting
rational singularities under smooth morphisms in
Proposition~\ref{p:rational-sings}, and an explicit construction of
such smooth morphisms via Bott--Samelson resolutions of singularities
in Proposition~\ref{p:submersion}.

\section{Applications to positivity conjectures}\label{s:applications}

\begin{cor}\label{c:GrKu}
Fix a positive basis $\beta_1,\ldots,\beta_r$ for a torus $S \subseteq
T$ acting on $X = G/P$ (Section~\ref{sub:full}).  For any $S$-stable
subvariety $Y \subseteq X$ of~$X$ with rational singularities,
$$%
  [\OO_Y] = \sum_{w \in W^P} a_w \OO_w \quad \text{with}\quad
  (-1)^{\dim Y - \ell(w)} a_w \in \NN[e^{-\beta_i} - 1]_{i=1}^r
  \quad\text{for all } w \in W^P.
$$
Write $e^{-\alpha_i}_S$ for the image in $R(S)$ of $e^{-\alpha_i} \in
R(T)$.  If the basis $\beta_1,\ldots,\beta_r$ is full
(Section~\ref{sub:full}), then positivity for $[\OO_Y]$ holds with
$\NN[e^{-\alpha_i}_S - 1]_{i=1}^n$ in place of\/
\mbox{$\NN[e^{-\beta_i} - 1]_{i=1}^r$}.
\end{cor}  

\begin{proof}
Apply Proposition~\ref{p:approx} and Lemma~\ref{l:monomials} to the
statement of Theorem~\ref{t:main}, noting that $(e^{-\beta_i} - 1)^J =
(-1)^{|J|}(1 - e^{-\beta_i})^J$.  When the basis is full,
every monomial in $1 - e^{-\beta_1},\ldots,1 - e^{-\beta_r}$ is a
monomial in $1 - e^{-\alpha_1}_S,\ldots,1 - e^{-\alpha_n}_S$ by
definition.
\end{proof}

The previous corollary is a special case of
\cite[Conjecture~7.1]{GrKu}.  It suffices for the applications to
Schubert calculus, such as the following; we do not know if our
methods extend to handle the general case, where the subtorus~$S$ is
arbitrary.

\begin{cor}[{\cite[Conjecture~3.1]{GrKu}}]\label{c:GrKu'}
Let $X = G/P$.  Using the dual classes $\xi^w$
(Section~\ref{sub:Kflag}), the Laurent polynomials \mbox{$p_{uv}^w \in
R(T)$} defined~by
$$%
  \xi^u \xi^v = \sum_{w \in W^P} p_{uv}^w \xi^w
$$
have alternating coefficients when written in terms of the variables
$e^{-\alpha_i} - 1$:
$$%
  (-1)^{\ell(w)-\ell(u)-\ell(v)} p_{uv}^w \in
  \NN[e^{-\alpha_i}-1]_{i=1}^n.
$$
\end{cor}
\begin{proof}
As Graham and Kumar remark before their Conjecture~7.1, apply
Corollary~\ref{c:GrKu} to $X \times X$, with $S$ the diagonal subtorus
of $T\times T$ and $Y$ the diagonal embedding of~$X_w$; it is key here
that this $S$ possesses a full basis for its weight lattice.
\end{proof}

Corollary~\ref{c:GrKu'} is dual to a positivity conjecture, formulated
previously by Griffeth and Ram, for the structure constants with
respect to the opposite Schubert class basis.  There does not seem to
be a direct way to derive one conjecture from the other: the formulas
expressing one set of structure constants in terms of the other
involve M\"obius inversion and are not manifestly positive.

\begin{cor}[{\cite[Conjecture~5.1]{GrRa}}]\label{c:GrRa}
Let $X = G/P$.  The Laurent polynomials \mbox{$c_{uv}^w \in R(T)$}
defined by
$$%
  \OO^u \cdot \OO^v = \sum_{w \in W^P} c_{uv}^w \OO^w \quad\text{for }
  u,v \in W^P
$$
have alternating coefficients when written in terms of the variables
$1 - e^{-\alpha_i}$:
$$%
  (-1)^{\ell(w)-\ell(u)-\ell(v)} c_{uv}^w \in
  \NN[e^{-\alpha_i}-1]_{i=1}^n.
$$
\end{cor}
\begin{proof}
The coefficient $c_{uv}^w$ is the pushforward of the product
\mbox{$\OO^u \OO^v \xi_w \in K_T(X)$} to a point.  The equivariant
class $\OO^v \xi_w = \big[X^v_w(-X^v\cap\del X_w)\big]$ is that of a
reflexive sheaf on a Richard\-son variety, with $Y = X^v_w$ and $\del
= \del X_w$ satisfying the hypotheses of Theorem~\ref{t:main}.  Now
apply the results in Section~\ref{s:approx} to the statement of
Theorem~\ref{t:main} (with $S = T$ and $w$ there replaced by~$u$
here), noting that \mbox{$(e^{-\alpha_i} - 1)^J = (-1)^{|J|}(1 -
e^{-\alpha_i})^J$}.
\end{proof}

\begin{remark}
As pointed out in \cite[Proposition~3.13]{GrKu},
Corollary~\ref{c:GrRa} is equivalent to ``signless'' positvity for
products in the basis of dualizing sheaves: writing
$$%
  [\omega_{X^u}]\cdot [\omega_{X^v}] = \sum_{w\in W^P} d_{uv}^w 
  [\omega_{X^w}]\cdot [\omega_{G/P}],
$$
the Laurent polynomials $d_{uv}^w \in R(T)$ satisfy
$$%
  d_{uv}^w \in \NN[e^{\alpha_i}-1]_{i=1}^n.
$$
\end{remark}

\begin{remark}
The positivity results in Corollaries~\ref{c:GrKu'} and~\ref{c:GrRa}
hold when restricted to arbitrary subtori $S \subseteq T$, even though
we can only show Theorem~\ref{t:main} for subtori with positive bases
for their weight lattices.  The reason is simply that the statements
of the corollaries restrict without obstacle to arbitrary subtori,
regardless of the proofs of the corollaries.  In particular, sign
alternation in ordinary $K$-theory follows from these equivariant
results.
\end{remark}

\section{A group action on the mixing space}\label{s:group}

For the duration of this section, set $X = G/P$, and fix a positive
basis $\{\beta_1,\ldots,\beta_r\}$ for the weight lattice of a
subtorus $S \subseteq T$ (Section~\ref{sub:full}).
 
The mixing space functor applied to the quotient map $G \to X$, where $S$
acts on~$G$ by left multiplication, expresses the mixing space $\XX$
as the quotient of the principal $G$-bundle $\cG$ by the action of the
parabolic subgroup~$P$ on the right.

On the other hand, let $\GG = \EE_m S \times^S G$, with $S$ acting on
$G$ by conjugation.  
Since $S$ acts by group automorphisms, $\GG$ is a group scheme over
$\PP$ with fiber~$G$.  Indeed, the evident multiplication map
$$%
  (\EE_m S \times G)\times_{\EE_m S}(\EE_m S \times G) \to \EE_m S\times G
$$
descends to $\GG\times_\PP \GG \to \GG$; the inverse map and identity
section are defined similarly and satisfy appropriate commutative
diagrams.  
Moreover, the action of $G$ on itself by left multiplication induces an action of 
the group scheme $\GG$ on the principal bundle $\cG$, and hence on the mixing space $\XX$.  
Note, however, that $\GG$ itself is \emph{not} a principal bundle,
since there is no right action of~$G$.

Let $B=TU$ be the Levi decomposition of $B$, with $U$ the maximal
unipotent group in $B$, and consider the corresponding group scheme
$$%
  \BB = \EE_m S \times^S B \subseteq \GG 
$$
over~$\PP$, where again $S$ acts on $B$ by conjugation.

Let $\Gamma_0 = \Hom(\PP,\GG)$ be the group of global sections
of~$\GG$, i.e., the $\PP$-points of this group scheme.  Write
$\Gamma_0(\BB) = \Hom(\PP,\BB)$; this is a connected algebraic group
over $\CC$.  The following asserts that the group scheme $\BB$ is
``generated by sections''.  It requires that the basis
$\{\beta_1,\ldots,\beta_r\}$ be positive.

\begin{lemma}\label{l:sections}
For any $x\in\PP$ and $p\in\BB$ in the fiber over $x$, there is a
section $\gamma\in\Gamma_0(\BB)$ such that $p=\gamma(x)$.
\end{lemma}
\begin{proof}
Write $\TT = \EE_m S \times^S T$ and $\UU = \EE_m S \times^S U$ for
the corresponding groups over~$\PP$.  We may assume $p \in \TT$ or $p
\in \UU$.

A section of~$\TT$ is an $S$-equivariant map $\EE_m S \to T$.  Since
$S$ acts trivially on~$T$, this is the same as a map $\EE_m S/S = \PP
\to T$.  These are exactly the constant maps, since $\PP$ is
projective and $T$ is affine, so sections of~$\TT$ are identified
with~$T$; in particular, every point of every fiber of~$\TT$ is in the
image of some section.

Forgetting the group structure, upon fixing a parametrization for each
root subgroup $\UU$ becomes a vector bundle on~$\PP$ which splits as a
sum of line bundles: $\UU = \bigoplus_\alpha \OO(\alpha)$, where the
sum runs over the subset of positive roots that are non-trivial upon
restriction to $S$.  The positive roots $\alpha \in R^+$ restrict to
nonnegative integer linear combinations of $\beta_1,\dots,\beta_r$, by
our positivity assumption, and it follows that $\UU$ is generated by
sections as a vector bundle.  The lemma follows from this.
\end{proof}

The action of $(GL_{m+1})^r$ on $\PP$ induces a natural action on
$\Gamma_0(\BB)$, by precomposition with the projection to~$\PP$.

\begin{defn}\label{d:mixing}
The \emph{mixing group} is the semidirect product
$$%
  \Gamma = \Gamma_0(\BB) \rtimes (GL_{m+1})^r.
$$
\end{defn}
Thus there is an exact sequence \mbox{$1 \to \Gamma_0(\BB) \to \Gamma
\to (GL_{m+1})^r \to 1$}.  As a semidirect product of connected groups, 
$\Gamma$ is also a connected algebraic group.

\begin{lemma}\label{l:orbits}
The mixing group $\Gamma$ acts on the mixing space~$\XX$ of $X = G/P$
with finitely many orbits, the closure of each orbit being a
bundle~$\XX_w$ over~$\PP$ associated to some Schubert variety $X_w
\subseteq X$.
\end{lemma}
\begin{proof}
The action of $\Gamma_0(\BB)$ is clear, and $(GL_{m+1})^r$ acts via
its action on $\EE_m S$ (lifting the action on $\PP$).
Lemma~\ref{l:sections} implies that the fiber of a $\Gamma$-orbit over
a point $p \in \PP$ is a $B$-orbit.  The result follows from this and
the definition of~$\XX_w$.
\end{proof}

\begin{remark}
Let $Q \subseteq G$ be the parabolic subgroup generated by~$B$ and the
centralizer of~$S$, so the Levi decomposition of~$Q$ is $LU_Q$ with $L
= C_G(S)$ the centralizer of~$S$, and $U_Q$ the unipotent radical.
Let $\QQ = \EE_m S \times^S Q$ be the corresponding group scheme
over~$\PP$.  Then $\BB\subseteq \QQ \subseteq \GG$, and the above
discussion applies with $\BB$ replaced by $\QQ$, noting that $S$ acts
trivially on $L$.  The orbits of $\Gamma_0(\QQ)\rtimes(GL_{m+1})^r$
are the bundles associated to the $Q$-orbits in $X$.

One can show that $\Gamma_0(\QQ) = \Gamma_0$ is the largest group
generated by sections in the sense of Lemma~\ref{l:sections}.  If the
torus $S$ is regular, i.e., $C_G(S) = T$, then $Q = B$ and~$\QQ =
\BB$.
\end{remark}

\section{Generic homological transversality}\label{s:transverse}

This section reduces the computation of the coefficients from
Theorem~\ref{t:main} to an Euler characteristic using an equivariant
homological Kleiman transversality principle in
Theorem~\ref{t:kleiman}.  Again let $\XX$ be the mixing space of $X =
G/B$, with the action of~$\Gamma$.  By Lemma~\ref{l:orbits}, the orbit
closures of the $\Gamma$-action are the Schubert bundles~$\XX_v$.

\begin{lemma}\label{l:bundle-orbits}%
The coherent sheaves $\OO_{\XX^w_J}$ and $\OO_{\XX^w_J}(-\del)$
on~$\XX$ are homologically transverse to the orbit closures~$\XX_v$ of
the $\Gamma$-action on~$\XX$.
\end{lemma}
\begin{proof}
Consider the mixing spaces $\XX^w_J$ and~$\XX_v$.  These bundles
over~$\PP_J$ and~$\PP$, respectively, intersect in the bundle that is
the restriction to~$\PP_J$ of the mixing space~$\XX^w_v$ of a
Richardson variety.  All of these spaces are Cohen--Macaulay, and the
intersections are proper, so Section~\ref{sub:transversality} applies,
and we see that $\OO_{\XX^w_J}$ is homologically transverse to orbit
closures.  Similarly, $\del = \del\XX^w_J$ is Cohen--Macaulay and
intersects $\XX_v$ properly, so $\OO_\del$ is homologically transverse
to orbit closures.  The claim for $\OO_{\XX^w_J}(-\del)$ follows from
the exact sequence $0 \to \OO_{\XX^w_J}(-\del) \to \OO_{\XX^w_J} \to
\OO_\del \to 0$.
\end{proof}

\begin{thm}[Equivariant homological Kleiman transversality]\label{t:kleiman}
Fix notation as in Theorems~\ref{t:main} and~\ref{t:strong}, but
weaken the hypotheses to allow the Cohen--Macaulay subvariety $Y\!$
not to have rational singularities.  Then
$$%
  c_{J,w}^{\,Y} = \chi\big(\YY\,\cap\,\gamma.\XX^w_J, \OO_{\YY\,\cap\,
  \gamma.\XX^w_J}(-\del)\big)
  \quad\text{and}\quad
  d_{J,w}^{\,Y} = \chi\big(\YY_J\,\cap\,\gamma.\XX^w, \OO_{\YY_J\,\cap\,
  \gamma.\XX^w}(-\del_\gamma)\big).
$$
\end{thm}
\begin{proof}
By Lemma~\ref{l:bundle-duality}, we have
$$%
\Big\<\big[\OO_\YY\big], \big[\OO_{\XX^w_J}(-\del)\big]\Big\>
  = \Big\<\sum_{I,v} c_{I,v}^{\,Y} \OO^I\cdot\ol\OO_v,
    \big[\OO_{\XX^w_J}(-\del)\big]\Big\>
  = c_{J,w}^{\,Y}.
$$
On the other hand, Theorem~\ref{t:sierra} and
Lemma~\ref{l:bundle-orbits} guarantee that for general $\gamma \in
\Gamma$,
$$%
  \big[\OO_\YY\big] \cdot \big[\gamma\OO_{\XX^w_J}(-\del)\big] = 
  \big[\OO_{\YY \,\cap\, \gamma.\XX^w_J}(-\del)\big].
$$
Since $\Gamma$ is connected, we have 
$$%
\big\<\big[\OO_\YY\big],\, \big[\OO_{\XX^w_J}(-\del)\big]\big\>
= \chi\big(\big[\OO_\YY\big] \cdot
\big[\OO_{\XX^w_J}(-\del)\big]\big) = \chi\big(\big[\OO_\YY\big] \cdot
\big[\gamma\OO_{\XX^w_J}(-\del)\big]\big),
$$
and the theorem for $c_{J,w}^{\,Y}$
follows.  The proof for $d_{J,w}^{\,Y}$ is essentially the same.
\end{proof}

\section{Rational singularities}\label{s:rat-sings}

A \emph{desingularization} of a variety $X$ is a nonsingular
variety~$\wt{X}$ together with a proper birational map $f\colon  \wt{X} \to
X$.  As is well known, desingularizations exist for any complex
variety $X$.  Moreover, if $X$ is equipped with the action of an
algebraic group, the desingularization may be chosen so that the
action extends to $\wt{X}$ and the map $f$ is equivariant.  If
$D\subseteq X$ is a divisor (invariant for the group action), one can
also arrange that $f^{-1}D$ be a normal crossings divisor in $\wt{X}$.

If $X$ is a possibly non-reduced scheme, a desingularization of $X$ 
is a desingularization of the underlying variety $X_\mathit{red}$.

A variety $X$ has \emph{rational singularities} if $X$ is normal and
it has a desingularization $f\colon  \wt{X} \to X$ such that 
 $\RR^i f_*(\OO_\wt{X}) = 0$  for all $i > 0$. 
Equivalently, $X$ has rational singularities if it is Cohen--Macaulay
and
 $f_*\omega_\wt{X} \isom \omega_X$ 
for a desingularization \mbox{$f\colon \wt{X} \to X$}.  In fact, if either of these
conditions holds for some desingularization of $X$, then it holds for
all of them.

A morphism of nonsingular varieties $f\colon  X\to Y$ is \emph{smooth} if
the differential \linebreak $df_x \colon T_x X \to T_{f(x)} Y$ is surjective for all $x
\in X$.  (In differential geometry, this is the same as a submersion.)
A smooth morphism is flat (see e.g.\ \cite[III.10,
Theorem~3$'$]{Mum}), and is an open map.

Our proof of the vanishing result in Theorem~\ref{t:vanishing}
requires the following fact.

\begin{prop}\label{p:rational-sings}
Fix a nonsingular complex variety~$X$.  Let $W$ and $Y$ be varieties
with rational singularities, with morphisms $u\colon W\to X$ and $v\colon Y\to X$.
Let \mbox{$\phi\colon  \wt{W} \to W$} and \mbox{$\psi\colon \wt{Y} \to Y$} be
desingularizations.  If\/ $W \to X$ is flat with reduced fibers, and $\wt{W} \to X$ is smooth, then  
\mbox{$\wt{W} \times_X \wt{Y} \to W \times_X Y$} is a desingularization and 
$W \times_X Y$ has rational singularities.
\end{prop}
\begin{proof}
Since $\wt{W}$ and $\wt{Y}$ are nonsingular and $\wt{W} \to X$ is a
smooth map, $\wt{W} \times_X \wt{Y}$ is nonsingular.  Since $\phi$ and
$\psi$ are proper, so is $\phi \times \psi\colon\wt{W} \times_X \wt{Y} \to
W \times_X Y$.  Birationality follows from that of
$\phi$ and~$\psi$, using the openness of $\wt{W} \to X$.

Next we observe that $\wt{W} \to W$ is a simultaneous resolution over $X$, in the sense of \cite{Elk}; that is, the maps $\wt{W}_x \to W_x$ are desingularizations for each $x$.  Indeed, for each $x$, the map $\wt{W}_x \to W_x$ is a proper, surjective morphism of reduced varieties of the same dimension, with connected fibers (using normality of $W$ and Zariski's main theorem).  Thus it is generically bijective, and hence birational. 

Since $u\colon W \to X$ is flat, $X$ is nonsingular, $W$ has rational singularities, and $\wt{W} \to \nolinebreak W$ is a simultaneous resolution, \cite[Th\'eor\`eme~3]{Elk} says that each fiber $W_x$ has rational singularities.  Therefore, the same is true of each fiber of $W \times_X Y \to Y$.  Since $Y$ has rational singularities, it follows from \cite[Th\'eor\`eme~5]{Elk} that $W\times_X Y$ does, as well.
\end{proof}

\section{Bott--Samelson varieties}\label{s:bs}

We will need some basic facts about Bott--Samelson varieties.  With
the exception of Lemma~\ref{l:smooth} and
Proposition~\ref{p:submersion}, the following can be found in standard
references; see e.g.\ \cite[Chapter 13]{Jan} or \cite{Mag}.

Let $P_i = B s_i B \cup B$ be the minimal parabolic subgroup generated
by $B$ and $s_i$.  Let $\underline{w} =
(s_{i_1},s_{i_2},\ldots,s_{i_r})$ be a (not necessarily reduced) word
in the simple reflections.  The corresponding \emph{Bott--Samelson
variety} is
\begin{eqnarray*}
\wt{X}_{\underline{w}}
  &=& P_{i_1}\times^B P_{i_2} \times^B \cdots \times^B P_{i_r}\times^B\{pt\}
\\&=& (P_{i_1} \times P_{i_2} \times \cdots \times P_{i_r})/B^r,
\end{eqnarray*}
where $B^r$ acts by
$$%
  (b_1,b_2,\dots,b_r).(p_1,p_2,\dots,p_r) = (p_1 b_1^{-1},b_1 p_2
  b_2^{-1},\dots,b_{r-1} p_r b_r^{-1}).
$$
This is a nonsingular variety of dimension $r$, with $B$ acting by
left multiplication.  It comes with a $B$-equivariant map
$\wt{X}_{\underline{w}} \to X=G/B$, sending the class of
$(p_1,\ldots,p_r)$ to the coset $p_1 \cdots p_rB$; this map has
image~$X_w$, where $w$ is the Demazure product (obtained by using the
relations $s_i^2=s_i$ in place of $s_i^2=1$) of the reflections
$s_{i_1}, \ldots, s_{i_r}$.

When $\underline{w}$ is a reduced word for $w$, the map
$\wt{X}_{\underline{w}} \to X_w$ is a desingularization; if $w \in
W^P$ is a minimal length coset representative the same is true of the
map to $X_w \subseteq G/P$.  Fix such desingularizations by choosing a
reduced word for each $w\in W$, and simply write $\wt{X}_w$ for the
corresponding Bott--Samelson variety.

The desingularization map $\wt{X}_w \to X_w$ is an isomorphism over
the Schubert cell~$C_w$, identifying $\big\{[p_1,\ldots,p_\ell]\in
\wt{X}_w \mid p_j \not\in B \text{ for all }j\big\}$ with~$C_w$.  The
complement of~$C_w$ in $\wt X_w$ is the boundary divisor
$$%
  \del\wt X_w = \wt X_1 \cup \cdots \cup \wt X_\ell,
$$
where
$$%
  \wt X_j = \{[p_1,\ldots,p_\ell] \in \wt X_w \mid p_j\in B\}.
$$
Evidently, $\wt{X}_j$ is isomorphic to the Bott--Samelson variety
$\wt{X}_{\underline{w}(\hat\jmath)}$, where
$$%
  \underline{w}(\hat\jmath) =
  (s_{i_1},\ldots,\widehat{s_{i_j}},\ldots,s_{i_\ell});
$$
in particular, $\del\hspace{-.5pt}\wt X_{\hspace{-.51pt}w}$ is a
transverse union of smooth $B$-stable codimension~$1$
\mbox{subvarieties}.

\begin{lemma}\label{l:smooth}
The following map is smooth:
$$%
\begin{array}{rcl}
  B^- \times (P_{i_1}\times\cdots\times P_{i_r})/B^r &  \too  & G/P
\\                       \big(b,(p_1,\dots,p_r)\big) &\mapstoo& bp_1\cdots p_r P.
\end{array}
$$
\end{lemma}
\begin{proof}
Consider, for $q \geq 1$, the map
\begin{equation}\label{multmap}
  B^- \times P_{i_1} \times \cdots \times P_{i_q} \to G
\end{equation}
given by multiplication.  When $q=1$ its differential is surjective
because of the following: the domain is a homogeneous space for $B^-
\times P_{i_1}$, with action $(b,p).(b',p')=(bb',p'p^{-1}))$; the map
is equivariant for the natural action of $B^- \times P_{i_1}$ on the
domain and target; and
$\text{Lie}(G)=\text{Lie}(B^-)+\text{Lie}(P_{i_1})$.  For $q > 1$ we
use induction.  The map~\eqref{multmap} can be written as the
composition of two multiplication maps
$$%
  B^- \times P_{i_1} \times\cdots\times P_{i_q} \too G \times P_{i_q} \too G.
$$
By induction the differential of the first map is surjective at all
points of the domain, and the second map obviously has the same
property.  It follows that \eqref{multmap} also has surjective
differential everywhere.  Upon composing \eqref{multmap} with the
projection from $G$ to $G/P$ we see that
$$%
\begin{array}{rcl}
  f\colon B^- \times P_{i_1}\times\cdots\times P_{i_q} &  \too  & G/P
\\                   \big(b,(p_1,\dots,p_q)\big) &\mapstoo& bp_1\cdots p_q P
\end{array}
$$
has surjective differential everywhere.  On the other hand $f$ factors
through the map $B^- \times \wt{X}_{w} \to G/P$, proving our
claim.
\end{proof}

The \emph{opposite Bott--Samelson varieties} $\wt{X}^{\underline{w}}$
are defined similarly.  To be precise, given a word
$\underline{w}=(s_{i_1},\ldots,s_{i_r})$ set
$$%
  \wt{X}^{\underline{w}} = P_{i_1}^-\times^{B^-} P_{i_2}^-
  \times^{B^-} \cdots \times^{B^-} P_{i_r}^-\times^{B^-}\{pt\},
$$
where $B^-$ is the opposite Borel, and the $P_i$ are the opposite
minimal parabolics.  This maps to $X=G/B$ via
$(p_{i_1},\ldots,p_{i_r}) \mapsto p_{i_1}\cdots p_{i_r} w_\circ B$.
When $\underline{w}$ is a reduced word for $ww_\circ$,
$\wt{X}^{\underline{w}} \to X^w$ is a resolution of the opposite
Schubert variety.  As before, fix desingularizations $\wt{X}^w$ for
opposite Schubert varieties by choosing reduced words.
Lemma~\ref{l:smooth} applies to opposite Bott--Samelson varieties,
exchanging $B$ and $B^-$.

Let $\wt\XX^{\underline{w}}$ be the approximate mixing space bundle
over~$\PP$ corresponding to $\wt{X}^{\underline{w}}$, and let
$\wt\XX^{\underline{w}}_J$ be its restriction to~$\PPJ$
(Section~\ref{sub:restriction}).

\begin{prop}\label{p:submersion}
The map $\Gamma \times \wt\XX^{\underline{w}}_J \xrightarrow{\wt{m}}
\XX$ is smooth.
\end{prop}
\begin{proof} 
Since the map in question is a map of fiber bundles
\begin{diagram}
\Gamma_0(\BB)\times\wt X^{\underline{w}}
                          &\rTo&
                  \Gamma\times\wt\XX^{\underline{w}}_J&\rTo&(GL_{m+1})^r\times\PPJ\\
           \dTo_{m'}      &    &   \dTo_{\wt m}       &    &\dTo_{m''}\\
               X          &\rTo&        \XX           &\rTo&\PP,
\end{diagram}
and smoothness is local on the source and on the target, it suffices
to prove that $m'$ and~$m''$ are smooth.  It is easy to see $m''$ is
smooth: indeed, $(GL_{m+1})^r$ acts transitively on $\PP$ and $m''$ is
equivariant, so it is a locally trivial fiber bundle with smooth
fiber.

The group $\Gamma_0(\BB) = \Hom(\PP,\BB)$ acts on the fiber over
$x\in\PP$ by first evaluating at~$x$, via a surjective group
homomorphism $\Gamma_0(\BB) \to B$.  Therefore the map $m'$ factors
through $B \times \wt{X}^{\underline{w}} \to X$, and the latter map is
smooth by Lemma~\ref{l:smooth} applied to opposite Bott--Samelson and
Schubert varieties.  Since the group homomorphism from $\Gamma_0(\BB)
\rightarrow B$ has surjective differential everywhere, $m'$ is also
smooth.
\end{proof}

\begin{prop}\label{p:flat}
The map $\Gamma \times \XX^w_J \rightarrow
\XX$ is flat and has normal fibers.
\end{prop}
\begin{proof}
As in the proof of Proposition~\ref{p:submersion} it suffices to prove
that the action map $B \times X^w \xrightarrow{a} X$ is flat and has
normal fibers.

We begin with flatness.  Note that its image is the union of the
Schubert cells $C_v$ such that $v \geq w$; this is an open subset $U$
of $X$.  Since $U$ is nonsingular and $X^w$ is Cohen--Macaulay, by
\cite[Exercise~III.10.9]{Har} it suffices to show that the non-empty
fibers of $m$ have constant dimension equal to $\dim(B \times
X^w)-\dim(X)$.  We will now show that that the fibers of $B \times
\wt{X}^w \xrightarrow{\wt{a}} X$ map birationally to the fibers of
$a$; since $\wt{a}$ is smooth this completes the proof.  Note that the
image of the restriction $B \times C^w \rightarrow X$ of $a$ contains
$C_v$ for all $v \geq w$ since $C_v \cap C^w \neq \emptyset$ if $v
\geq w$.  Therefore every non-empty fiber of $a$ meets $B \times C^w$,
and hence is birational to the corresponding fiber of $\wt{a}$.

For normality of the fibers, observe that by $B$-equivariance, the
fibers of the action map $a$ over the points in $C_v$ are all
isomorphic to the fiber over $vP \in X$.  Write
$B_v=\text{Stab}_B(vP)$ for the stabilizer in $B$ of $vP$, and $U_v$
for the subgroup of $B$ generated by those root subgroups not
stabilizing $vP$; thus the action map gives an isomorphism $U_v \cong
C_v$.  Let $h\colon C_v \rightarrow U_v$ be the inverse isomorphism, with
$h(uvP)=u$.  Then one checks by computing its inverse that the map
$B_v \times (C_v \cap X^w) \rightarrow B \times X^w$ given by $(b,x)
\mapsto (b h(x)^{-1},x)$ is an isomorphism onto the fiber over $vP$,
\mbox{which is therefore normal}.
\end{proof}

\section{A vanishing theorem for flag bundles}\label{s:vanishing}

In this final section, we prove the vanishing theorems required to
complete the proof of Theorem~\ref{t:main}.  We need some preliminary
results.

\begin{lemma}[{\cite[Lemma, page~108]{FuPr}}]\label{l:FP}
Let $f\colon  \WW \to \XX$ be a morphism from a pure-dimensional
scheme~$\WW$ to a nonsingular variety~$\XX$, and let $\YY$ be a
Cohen--Macaulay closed subscheme of~$\XX$.  Set $Z = f^{-1}(\YY)$.
If~$\WW$ is Cohen--Macaulay and $\codim(Z,\WW) \geq \codim(\YY,\XX)$,
then equality holds and $Z$ is Cohen--Macaulay.\qed
\end{lemma}

Recall the spaces $\mathcal{X}_J^w$ from Section~\ref{sub:restriction}
and the notation from Section~\ref{s:transverse}, and consider the diagram
\begin{equation}\label{eq:Z}
\begin{diagram}
         &           &           Z          & \rTo^\mu  &  \YY   \\
         & \ldTo^\pi &     \dInto_\iota     & \boxtimes & \dInto \\
  \Gamma &   \lTo    & \Gamma\times \XX^w_J &  \rTo_m   &  \XX
\end{diagram}
\end{equation}
in which $Z$ is defined by the fiber square.  Note that
$$%
  \pi^{-1}(\gamma) \isom \YY \cap \gamma.\XX^w_J.
$$
When this fiber is nonempty for general $\gamma\in\Gamma$, it is
nonempty for all $\gamma$, so $\pi$ is surjective.  We shall assume
surjectivity below, since all the statements are trivial if $Z$ is
empty.

\begin{lemma} \label{l:z-rational-sings}
With notation as above, if $Y$ has rational singularities then $Z$
does, too.  In particular, $Z$ is Cohen--Macaulay, so it has a
dualizing sheaf $\omega_Z$.
\end{lemma}

\begin{proof}
Let $\wt{Y} \to Y$ be
an $S$-equivariant desingularization of~$Y$, and let $\wt{X}^w \to
X^w$ be the Bott--Samelson desingularization of~$X^w$, which is also
$S$-equivariant.  Let $\phi\colon \wt\YY \to \nolinebreak \YY$ and $\psi\colon \wt\XX^w_J \to
\XX^w_J$ be the induced desingularizations of bundles, and define
notation by the diagram
\begin{equation}\label{eq:wZ}
\begin{diagram}
         &                  &          \wt Z          & \rTo^{\wt\mu} & \wt\YY\\
         & \ldTo^{\wt{\pi}} &   \dTo_{\tilde\iota}    &   \boxtimes   &  \dTo \\
  \Gamma &       \lTo       & \Gamma\times \wt\XX^w_J & \rTo_{\wt m}  &  \XX
\end{diagram}
\end{equation}
mapping to \eqref{eq:Z}.  By Propositions~\ref{p:submersion}, \ref{p:flat}, and~\ref{p:rational-sings}, the
map $f\colon \wt{Z} \to Z$ is a desingularization, and $Z$ has rational singularities.  (The
maps to $\Gamma$ do not arise until the proof of
Theorem~\ref{t:vanishing}.)  
\end{proof}

The proof of Lemma~\ref{l:z-rational-sings} also shows the following.

\begin{lemma}\label{l:dim}
For general $\gamma \in \Gamma$,
$$%
  \dim(\YY \,\cap\, \gamma.\XX^w_J) = \dim\YY - \codim\XX^w_J
  = \dim Y - \ell(w) + |J|.
$$
\end{lemma}
\begin{proof}
A general fiber of $\pi\colon Z \to \Gamma$ has dimension
\begin{eqnarray*}
\dim Z - \dim \Gamma
  &=& \dim \YY + \dim ( \Gamma \times \XX^w_J ) - \dim\XX - \dim\Gamma \\
  &=& \dim\YY + \dim \XX^w_J - \dim\XX \\
  &=& \dim\YY - \codim(\XX^w_J,\XX),
\end{eqnarray*}
as claimed.
\end{proof}

Since sheaf cohomology can only be nonzero in cohomological degrees
between zero and the dimension of the ambient scheme, the following
vanishing theorem places the final nails in the proof of
Theorem~\ref{t:main}.  
Parts~1 and~2 are, respectively, the statements
needed for positivity of the $c$~coefficients and the
$d$~coefficients.  Part~1 is based on the diagram \eqref{eq:Z}, where
$Z$ has a boundary divisor arising from a given boundary on $\Gamma
\times \XX^w_J$.  Part~2 simply swaps the roles of $\Gamma \times
\XX^w$ and~$\YY$: the boundary divisor on~$Z$ is pulled back from the
boundary of~$\YY_J$, which also carries the restriction to~$\PP_J$.

\begin{thm}\label{t:vanishing}
Assume the hypotheses and notation from Theorems~\ref{t:main}
and~\ref{t:strong}, including the hypothesis that\/ $Y$ has rational
singularities.  Fix a general element~\mbox{$\gamma \in \Gamma$}.  
\begin{enumerate}[1.]
\item
For all $w \in W$ and $i < \dim(\YY \cap \gamma.\XX^w_J) = \dim Y -
\ell(w) + |J|$,
$$%
  H^i\big(\YY \cap \gamma.\XX^w_J, \OO(-\del)\big) = 0.
$$
Equivalently, for all $w \in W$ and $i > 0$,
$$%
  H^i\big(\YY \cap \gamma.\XX^w_J,
  \omega_{\YY\cap\gamma.\XX^w_J}(\del)\big) = 0.
$$
\item
For all $w \in W$ and $i < \dim(\YY_J \cap \gamma.\XX^w) = \dim Y -
\ell(w) + |J|$,
$$%
  H^i\big(\YY_J \cap \gamma.\XX^w, \OO(-\del_\gamma)\big) = 0.
$$
Equivalently, for all $w \in W$ and $i > 0$,
$$%
  H^i\big(\YY_J\cap\gamma.\XX^w,\omega_{\YY_J\cap\gamma.\XX^w}(\del_\gamma)\big) = 0.
$$
\end{enumerate}
\end{thm}
\begin{proof}
The statements beginning ``Equivalently'' follow from Serre duality,
using the fact that $\YY\cap\gamma.\del\XX^w_J$ and
$\del\YY_J\cap\gamma.\XX^w$ are Cohen--Macaulay to get degeneration of
the local-to-global spectral sequence (cf.\ \cite[Lemma~4]{brion}).

The rest of the proof follows that of \cite[Theorem~3]{brion}.  We
will assume \mbox{$X = G/B$} until the very end of this section; in
fact, the entire proof works verbatim for general $G/P$ except the
verification of Corollary~\ref{c:wtZvanishing}.

Recall the notation defined by the diagram \eqref{eq:Z}.  Define the
boundary divisor
$$%
  \del Z = \YY \times_\XX (\Gamma \times \del \XX^w_J)
$$
of~$Z$.  For general $\gamma \in \Gamma$, we have
$$%
  \omega_Z(\del)|_{\pi^{-1}(\gamma)} \isom
  \omega_{\YY\cap\gamma.\XX^w_J}(\del),
$$
so it will suffice to prove that
\begin{equation}\label{eqn:z-vanishing}
  R^i \pi_*\omega_Z(\del) = 0 \quad\text{for } i > 0.
\end{equation}
We shall accomplish this by applying the Kawamata--Viehweg theorem, in
the form of Theorem~\ref{t:kv}, to the desingularization of $Z$
constructed in the proof of Lemma~\ref{l:z-rational-sings}.

Recall the diagram \eqref{eq:wZ}.  We have seen that $f\colon \wt{Z} \to Z$
is a desingularization, and $Z$ has rational singularities.  Let
$\wt\XX_1,\ldots,\wt\XX_\ell$ be the bundles over $\PP$ corresponding
to the components of the boundary divisor $\del \wt X^w = \wt X_1
\cup\cdots\cup \wt X_\ell$.

\begin{lemma}\label{l:ample}
The boundary divisor
$$%
  \del\wt\XX^w_J = \wt\XX^w|_{\del\PPJ} \cup \bigcup_{i=1}^\ell
  \wt\XX_i|_\PPJ
$$
of $\wt\XX^w_J$ supports an ample line bundle on~$\wt\XX^w_J$.
\end{lemma}
\begin{proof}
This follows in a straightforward manner from Example~\ref{ex:rho} by
pulling back very ample line bundles.  The details are omitted.
\end{proof}

The divisor in Lemma~\ref{l:ample} gives rise to a boundary divisor
$$%
  \del\wt Z = \wt\YY \times_\XX (\Gamma \times \del \wt\XX^w_J)
$$
that is a union of nonsingular irreducible divisors intersecting
transversally---that is, with normal crossings---by
Proposition~\ref{p:submersion} applied to the components of $\del
\wt\XX^w_J$ and all of their intersections, each of which is still a
Bott--Samelson fibration.

Our next goal is to prove vanishing on~$\wt{Z}$.
\begin{prop}\label{p:wtZvanishing}
$\displaystyle R^i \wt\pi_*\omega_{\wt{Z}}(\del) = 0$ for $i > 0$.
\end{prop}
\begin{proof}
For this, let $b_0\wt\XX_0 + b_1\wt\XX_1 + \cdots + b_\ell\wt\XX_\ell$
be the divisor of a very ample line bundle supported on
$\del\wt\XX^w_J = \bigcup_{i=0}^\ell \wt\XX_i$, as in
Lemma~\ref{l:ample}, and let $\wt{Z}_i = \wt\YY \times_\XX (\Gamma
\times \wt\XX_i)$.  Fix an integer~$N$ greater than every~$b_i$, and
write $a_i = N-b_i$.  Set $\cM = \OO_{\wt{Z}}(\del\wt{Z})$ and $D =
a_0 \wt{Z}_0 + \cdots + a_\ell \wt{Z}_\ell$.  Then
\begin{align*}
\cM^{\otimes N}(-D)
  &= \OO_{\wt{Z}}(b_0\wt{Z}_0 + \cdots + b_\ell\wt{Z}_\ell)
\\&= \wt\iota^*\OO_{\Gamma\times\wt\XX^w_J}\big(\Gamma \times
     (b_0\wt\XX_0+\cdots+b_\ell\wt\XX_\ell)\big)
\end{align*}
is the pullback under the map $\wt\iota$ of a very ample sheaf on
$\Gamma\times\XX^w_J$, so it is nef (i.e.,~its intersection with every
curve is nonnegative).  In particular, it is $\wt\pi$-nef and $f$-nef.
It is $\wt\pi$-big, because a general fiber $\wt\pi^{-1}(\gamma) =
\wt\YY \times \gamma.\wt\XX^w_J$ maps birationally onto its image
under $\wt\iota$.  This verifies the hypotheses of Theorem~\ref{t:kv},
whose conclusion says that $R^i\wt\pi_*(\cM\otimes\omega_{\wt{Z}})=0$
for $i > 0$, concluding the proof of the proposition.
\end{proof}

\begin{cor}\label{c:wtZvanishing}
$R^i f_*\omega_{\wt{Z}}(\del\wt{Z}) = 0$ for $i > 0$.
\end{cor}
\begin{proof}
Continuing notation as in the proof of
Proposition~\ref{p:wtZvanishing}, $\cM^{\otimes N}(-D)$ is $f$-nef and
$f$-big, the latter because $f$ is birational.
\end{proof}

The final constituent in the proof of part~1 is the following.
\begin{prop}\label{p:f-equality}
$f_*\omega_{\wt{Z}}(\del\wt{Z}) = \omega_Z(\del Z)$.
\end{prop}
\begin{proof}
Consider the factorization of $f\colon \wt{Z}\to Z$ given by
\begin{equation}
\begin{diagram}\label{eq:f-prime}
\wt{Z} & \rTo^{f'}      &  Z'         & \rTo^{\phi'} & Z        \\
\mbox{}& \rdTo_{\wt\mu} & \dTo_{\mu'} & \ \boxtimes  & \dTo_\mu \\
       &                & \wt\YY      & \rTo_\phi    & \YY
\end{diagram}
\end{equation}
in which the $\boxtimes$ denotes a fiber square.  Note that all fibers of 
the flat morphism $\mu'$ are normal, since they are the same as those of 
$m\colon \Gamma\times\XX^w_J \to \XX$, which are normal by Proposition \ref{p:flat}.  
Therefore $Z'$ is normal by \cite[Corollary 23.9]{Matsumura}.

We will establish the following:
\begin{align}
\label{eqn:z-tilde-dualizing}
\omega_{\wt{Z}}(\del\wt{Z}) &\isom
\wt\mu^*(\omega_{\wt\YY}\otimes\phi^*e^\rho\cL_\rho(c\cdot\del\PP))
\\
\label{eqn:z-dualizing}
\omega_{Z}(\del Z) &\isom
\mu^*(\omega_{\YY}\otimes e^\rho\cL_\rho(c\cdot\del\PP)),
\end{align}
where $c = (c_1,\ldots,c_r)$ is a multi-index, with
$$%
c_i =
  \begin{cases}
   m-j_i+1 &\text{if } j_i<m
  \\
   0 &\text{if } j_i = m
  \end{cases}
$$
being the coefficient of the corresponding component of~$\del\PP$, so
that $\omega_{\PP^J} \isom \OO_{\PP^J}(-c\cdot\del)$.  Granting these
isomorphisms for the moment, we have
\begin{align*}
f'_*\omega_{\wt{Z}}(\del\wt{Z})
  &= f'_* f'^* \mu'^*\left(\omega_{\wt\YY} \otimes
     \phi^*e^\rho\cL_\rho(c\cdot\del\PP)\right)
\\&= \mu'^*\left(\omega_{\wt\YY} \otimes
     \phi^*e^\rho\cL_\rho(c\cdot\del\PP)\right),
\end{align*}
using the projection formula and the fact that $f'_*\OO_{\wt{Z}} =
\OO_{Z'}$, since $Z'$ is normal.  Therefore
\begin{align*}
f_*\omega_{\wt{Z}}(\del\wt{Z})
  &= \phi'_*\mu'^*\left(\omega_{\wt\YY} \otimes
     \phi^*e^\rho\cL_\rho(c\cdot\del\PP)\right)
\\&= \mu^*\phi_*\left(\omega_{\wt\YY} \otimes
     \phi^*e^\rho\cL_\rho(c\cdot\del\PP)\right),
\end{align*}
because $\mu$ is flat (by Proposition \ref{p:flat}).  Finally, using
the projection formula, rational singularities of $\YY$,
and~\eqref{eqn:z-dualizing}, we obtain
\begin{align*}
f_*\omega_{\wt{Z}}(\del\wt{Z})
  &= \mu^*\left(\omega_\YY \otimes e^\rho\cL_\rho(c\cdot\del\PP)\right)
\\&= \omega_Z(\del Z),
\end{align*}
proving the proposition.

It remains to check \eqref{eqn:z-tilde-dualizing} and
\eqref{eqn:z-dualizing}.  The morphism $\wt\mu$ is smooth, by
\eqref{eq:wZ} and Proposition~\ref{p:submersion}.  Therefore,
$$%
  \omega_\wt{Z} \isom \wt\mu^*\omega_\wt\YY\otimes\omega_{\wt{Z}/\wt\YY}.
$$
Moreover, since the projection $\wt\XX^w_J \to \PP_J$ is a locally
trivial fibration, $\omega_{\wt\XX^w_J/\PP_J}$ is isomorphic to the
line bundle on $\wt\XX^w_J$ induced by the equivariant line bundle
$\omega_{\wt{X}^w}$ on the (non-mixing space) Bott--Samelson variety
$\wt{X}^w$.  The latter bundle is $e^{\rho}\cL_{-\rho}\otimes
\OO_{\wt{X}^w}(-\del)$,~so
\begin{eqnarray*}
\omega_{\wt\XX^w_J/\PP_J}
  &\isom&e^\rho\cL_{-\rho}\otimes\OO_{\wt\XX^w_J}(-\wt\XX_1-\cdots-\wt\XX_\ell)
\\&\isom&e^\rho\cL_{-\rho}\otimes\OO_{\wt\XX^w_J}(-\del\wt\XX^w_J+\del\PP_J)
\end{eqnarray*}

Finally, using the formula $\omega_X \isom \cL_{-2\rho}$ (here we use
$X=G/B$) and
suppressing notation for some obvious pullbacks, we have
\begin{align*}
\omega_{\wt{Z}/\wt\YY}
  &= \wt\iota^*\omega_{(\Gamma\times \wt\XX^w_J)/\XX}
\\&= \wt\iota^*\hspace{-1.5pt}\big(\omega_{\Gamma\times\wt\XX^w_J}\otimes
     \wt{m}^*\omega_\XX^{-1} \big)
\\&= \wt\iota^*\hspace{-1.5pt}\big(\omega_\Gamma \otimes
     \omega_{\wt\XX^w_J/\PP_J}\otimes\omega_{\PP_J} \otimes
     \wt{m}^*\omega_{\XX/\PP}^{-1} \otimes \wt{m}^*\omega_{\PP}^{-1}
     \big)
\\&= \wt\iota^*\hspace{-1.5pt}\big(\OO_{\Gamma}
     \hspace{-1.5pt}\otimes\hspace{-1.5pt} \wt{m}^*e^\rho\cL_{-\rho}
     \hspace{-1.5pt}\otimes\hspace{-1.5pt}
     \OO_{\wt\XX^w_J}(-\del\wt\XX^w_J\hspace{-1.5pt}
                      +\hspace{-1.5pt}\del\PP_{\!J}\hspace{-1pt})
     \hspace{-1.5pt}\otimes\hspace{-1.5pt}
     \OO_{\PP_{\!J}}(-c'\hspace{-.75ex}\cdot\hspace{-.5ex}\del\PP_{\!J}\hspace{-1pt})
     \hspace{-1.5pt}\otimes\hspace{-1.5pt}\wt{m}^*\cL_{2\rho}
     \hspace{-1.5pt}\otimes\hspace{-1.5pt}
     \wt{m}^*\OO_{\PP}(c''\hspace{-.8ex}\cdot\hspace{-.5ex}\del\PP)
     \hspace{-1pt}\big)
\\&= \wt\iota^*\hspace{-1.5pt}
     \big(\OO_{\Gamma\times\wt\XX^w_J}(-\del\wt\XX^w_J)
     \otimes \wt{m}^*e^\rho\cL_\rho \otimes
     \wt{m}^*\OO_{\PP}((1_J-c'+c'')\cdot\del\PP) \big)
\\&= \OO_{\wt{Z}}(-\del\wt{Z}) \otimes
     \wt\iota^*\wt{m}^*e^\rho\cL_\rho(c\cdot\del\PP)
\\&= \OO_{\wt{Z}}(-\del\wt{Z}) \otimes
     \wt\mu^*\phi^*e^\rho\cL_\rho(c\cdot\del\PP).
\end{align*} 
Here
\begin{eqnarray*}
c'_i &=&
  \begin{cases}
   j_i+1 &\text{if } j_i>0
  \\
   0 &\text{if } j_i = 0,
  \end{cases} \\
(1_J)_i &=&
  \begin{cases}
   1 &\text{if } j_i>0
  \\
   0 &\text{if } j_i = 0,
  \end{cases}
\end{eqnarray*}
and $c'' = (m+1,\ldots,m+1)$, so $\OO_{\PP_J}(-c'\cdot\del\PP_J) \isom
\omega_{\PP_J}$ and $\OO_{\PP}(-c''\cdot\del\PP) \isom \omega_{\PP}$.
Thus $1_J - c' + c'' = c$.

This proves \eqref{eqn:z-tilde-dualizing}.  As in Brion's proof,
\eqref{eqn:z-dualizing} is proved similarly, by restricting to the
smooth locus of the normal variety $Z$.  Thus
Proposition~\ref{p:f-equality} is proved.
\end{proof}

Proposition~\ref{p:f-equality}, together with
Proposition~\ref{p:wtZvanishing}, Corollary~\ref{c:wtZvanishing}, and
the Leray spectral sequence for $\wt\pi = \pi\circ f$,
implies~\eqref{eqn:z-vanishing}, which completes the proof of part~1.

For part~2, the restriction to $\PPJ$ now appears as~$\YY_J$ instead
of~$\XX^w_J$.  Using $\boxtimes$ to denote a pullback square, define
notation by the diagram
\begin{equation}\label{eq:Z2}
\begin{diagram}
         &           &           Z        &\rTo^{\mu\ \ }&  \YY_J \\
         & \ldTo^\pi &     \dInto_\iota   &\!\!\boxtimes & \dInto \\
  \Gamma &   \lTo    & \Gamma\times \XX^w &   \rTo_m     &  \XX
\end{diagram}
\end{equation}
and let $\del Z = \del\YY_J \times_\XX (\Gamma\times\XX^w)$.
Lemma~\ref{l:z-rational-sings} holds verbatim in this notation, with
the same proof, mutatis mutandis.  Similarly, the analogue of
Lemma~\ref{l:dim} still holds.  The proof of vanishing for part~2,
however, is somewhat different from the proof of~part~1.

Since $\mu$ is flat, we have $\OO_Z(\del Z) =
\mu^*\OO_{\YY_J}(\del\YY_J)$.  Choose (using \cite{Vil92}, say) an
$S$-equivariant resolution of singularities $\wt Y \to Y$ so that
$\del\wt\YY = \phi^{-1}\del\YY$ is a normal crossings divisor with
ideal sheaf $\II(\del\YY)\cdot\OO_{\wt\YY}$, where $\phi\colon \wt\YY \to
\YY$.  The analogue of~\eqref{eq:wZ} has the subscript $J$ on~$\wt\YY$
instead of~$\wt\XX^w$; using that notation, the divisor $\del\wt{Z} =
\del\wt\YY_J \times_\XX (\Gamma\times\wt\XX^w)$ also has normal
crossings, and $\OO_{\wt{Z}}(\del\wt{Z}) =
\wt\mu^*\OO_{\wt\YY}(\del\wt\YY)$.  Write $\del\wt{Z} =
D_1+\cdots+D_\ell$, with each $D_i$ a (nonsingular) irreducible
component.

The sheaf $\OO_Z(\del)$ is reflexive of rank $1$ on the normal
variety~$Z$.  Writing $f\colon  \wt Z \to Z$,
$$%
  \shfHom(f^*\OO_Z(-\del),\OO_{\wt{Z}}) =
  f^*\OO_Z(\del)/\mathit{torsion} =
  f^* \mu^*  \OO_{\YY_J}(\del \YY)/\mathit{torsion}
$$
is a reflexive rank $1$ sheaf~$\cM$ on the smooth variety $\wt{Z}$;
therefore it is a line bundle.  Note that $f_*\cM \isom \OO_Z(\del)$.
Since $\wt{\mu}$ is smooth and $\phi \wt{\mu}=\mu f$,
\begin{align*}
\OO_{\wt{Z}}(\del\wt{Z})&=\wt{\mu}^* \OO_{\wt\YY_J}(\phi^{-1} \del \YY)
=\wt{\mu}^* \big( (\phi^* \OO_{\YY_J}(\del \YY)) / \mathit{torsion} \big) \\
&=\big( (\phi \wt{\mu})^* \OO_{\YY_J}(\del \YY) \big)/ \mathit{torsion}=\cM.
\end{align*}
Since $Z$ has rational singularities,
$$%
  f_*(\omega_{\wt{Z}} \otimes \cM) =
  f_*\shfHom\!\big(f^*\OO_Z(-\del),\omega_{\wt{Z}}\big) = \omega_Z(\del).
$$
So it will suffice to prove
the analogues of Proposition~\ref{p:wtZvanishing} and
Corollary~\ref{c:wtZvanishing}.
\begin{lemma}
$R^i \wt\pi_*(\omega_{\wt{Z}}\otimes\cM) = 0$ for $i > 0$, and
$R^i f_*(\omega_{\wt{Z}}\otimes\cM) = 0$ for $i > 0$.
\end{lemma}
\begin{proof}
It suffices to find a divisor $D$ such that $\cM^{\otimes N}(-D)$ is
$\wt\pi$-big, $f$-big, $\wt\pi$-nef, and $f$-nef.  Fix an ample line
bundle~$\cL$ supported on~$\del\YY$.  Writing $\wt\mu^*\phi^*\cL \isom
\OO_{\wt{Z}}(b_1 D_1 +\cdots+ b_\ell D_\ell)$, let $a_i = N-b_i$ for
an integer $N$ greater than all the $b_i$'s, and let $D = a_1 D_1
+\cdots+ a_\ell D_\ell$.  Thus $\cM^{\otimes N}(-D) \isom
\wt\mu^*\phi^*\cL$.  This is nef, since it is a pullback of the ample
line bundle $\cL$; hence it is $\wt\pi$-nef and $f$-nef.  It is
$\wt\pi$-big and $f$-big for the same reasons as in part~1.  This
concludes the proof of the lemma, and with it part~2.
\end{proof}

To finish the proof of Theorem~\ref{t:vanishing}, it remains to treat
the general $G/P$ case, as opposed to the $G/B$ case we have been
assuming until now.  We proceed as in \cite[Lemma~4]{brion}.  As noted
earlier, our entire $G/B$ proof works verbatim for general $G/P$
except for the verification of Corollary~\ref{c:wtZvanishing}.  In
particular, the proof for part~2 is the same, so we may assume the
situation of part~1.

For the rest of this proof, write $X = G/P$ and $\hat{X}=G/B$, and
similarly for Schubert varieties and mixing spaces.  (Thus we have a
proper birational map $\hat{\XX}^w_J \to \XX^w_J$ of Schubert
varieties, with $w$ a maximal-length coset representative.)  Given an
$S$-invariant subvariety $Y\subseteq X$, let $\hat{Y}$ be its inverse
image in $\hat{X}$.  Note that the projection $\hat{X}\to X$ is a
locally trivial fiber bundle, with fiber $P/B$, so the same is true of
$\hat{\XX} \to \XX$, $\hat{\YY} \to \YY$, and $\hat{\YY}_J \to \YY_J$.
Define notation by the diagram
\begin{equation}\label{eq:Z-hat}
\begin{diagram}
         &           &           \hat{Z}     &\rTo^{\hat\mu\ \ }&  \hat\YY \\
         & \ldTo^{\hat\pi} &     \dInto_{\hat\iota}   &\!\!\boxtimes & \dInto \\
  \Gamma &   \lTo    & \Gamma\times \hat\XX^w_J &   \rTo_{\hat{m}}     &  \hat\XX,
\end{diagram}
\end{equation}
and let $\zeta\colon \hat{Z}\to Z$ be the induced map.  It is easy to see
that $\zeta$ is proper and birational, and in fact the resolution
$f\colon \wt{Z} \to Z$ factors as $f = \zeta\circ \hat{f}$, where
$\hat{f}\colon \wt{Z} \to \hat{Z}$ is the resolution for the $G/B$ case.
Since we know $\hat{f}_*\omega_{\wt{Z}}(\del\wt{Z}) =
\omega_{\hat{Z}}(\del\hat{Z})$, it will suffice to show that
\begin{equation}\label{z-hat-canonical}
  \zeta_*\omega_{\hat{Z}}(\del\hat{Z}) = \omega_Z(\del Z).
\end{equation}
For this, first note that $\zeta_*\OO_{\hat{Z}}=\OO_Z$, since $Z$ is
normal, and $\zeta^{-1}(\del Z) = \del\hat{Z}$ from the definitions.
Therefore $\zeta_*\OO_{\hat{Z}}(-\del) = \OO_Z(-\del)$.  Also, we have
$\zeta_*\omega_{\hat{Z}} = \omega_Z$, since $f$ and $\hat{f}$ are
rational resolutions, so $f_*\omega_{\wt{Z}} = \omega_Z$ and
$\hat{f}_*\omega_{\wt{Z}} = \omega_{\hat{Z}}$.  Now we compute:
\begin{eqnarray*}
 \zeta_*\omega_{\hat{Z}}(\del\hat{Z})
 &=& \zeta_*(\shfHom(\OO_{\hat{Z}}(-\del\hat{Z}),\omega_{\hat{Z}}) \\
 &=& \shfHom(\zeta_*\OO_{\hat{Z}}(\del\hat{Z}),\zeta_*\omega_{\hat{Z}}) \\
 &=& \shfHom(\OO_Z(\del Z), \omega_Z) \\
 &=& \omega_Z(\del Z).
\end{eqnarray*}
This proves \eqref{z-hat-canonical}, completing the proof of
Theorem~\ref{t:vanishing}.
\end{proof}

\raggedbottom


\begin{thebibliography}{EdGr98}

\bibitem[And07]{and}
Dave Anderson, \emph{Positivity in the cohomology of flag bundles (after
  Graham)}.  \textsf{arXiv: math.AG/0711.0983}

\bibitem[Bor91]{Bor}
Armand Borel, \emph{Linear Algebraic Groups}, Graduate Texts in
  Mathematics vol.~126, Springer-Verlag, 1991.

\bibitem[Bri02]{brion}
Michel Brion, \emph{Positivity in the Grothendieck group of complex flag
  varieties}, J. Algebra \textbf{258} (2002), 137--159.

\bibitem[Buc02]{buchLR}
Anders Skovsted Buch, \emph{A Littlewood--Richardson rule for the
  $K$-theory of Grassmannians}, Acta Math. \textbf{189} (2002), 37--78.

\bibitem[EdGr98]{EdGr}
Daniel Edidin and William Graham, \emph{Equivariant intersection
  theory}, Invent. Math. \textbf{131} (1998), 595--634.

\bibitem[EdGr00]{EdGrK}
Daniel Edidin and William Graham, \textit{Riemann--Roch for
  equivariant Chow groups}, Duke Math. J. \textbf{102} (2000), no.~3,
  567--594.

\bibitem[Ehr34]{Ehr}
Charles Ehresmann, \emph{Sur la topologie de certains espaces
  homog{\`{e}}nes}, Ann. of Math.~(2) \textbf{35} (1934), 396--443.

\bibitem[Elk78]{Elk}
Ren\'ee Elkik, \textit{Singularit\'es rationnelles et d\'eformations}, Invent. Math. {\bf 47} (1978), no. 2, 139--147. 

\bibitem[EsVi92]{EsVi}
H\'el\`ene Esnault and Eckart Viehweg, \emph{Lectures on vanishing
  theorems}, DMV Seminar vol.~20, Birkh\"auser, Basel, 1992.

\bibitem[Ful07]{eq}
William Fulton, \emph{Equivariant cohomology in algebraic geometry},
  lectures at Columbia University, notes by D.\ Anderson, 2007.  
  \textsf{http:/$\!$/www.math.osu.edu/\~{}anderson.2804/eilenberg}

\bibitem[FuLa94]{fulton-lascoux}
William Fulton and Alain Lascoux, \emph{A Pieri formula in the
  Grothendieck ring of a flag bundle}, Duke Math. J. \textbf{76}
  (1994), no.~3, 711--729.

\bibitem[FuPr98]{FuPr}
William Fulton and Piotr Pragacz, \emph{Schubert varieties and
  degeneracy loci}, Lecture Notes in Mathematics vol.~1689,
  Springer-Verlag, 1998.

\bibitem[Gra01]{graham}
William Graham, \emph{Positivity in equivariant Schubert calculus},
  Duke Math. J. \textbf{109} (2001), no.~3, 599--614.

\bibitem[GrKu08]{GrKu}
William Graham and Shrawan Kumar, \emph{On positivity in
  $T$-equivariant $K$-theory of flag varieties}.
  \textsf{arXiv:math.AG/0801.2776}

\bibitem[GrRa04]{GrRa}
Stephen Griffeth and Arun Ram, \textit{Affine Hecke algebras and the
  Schubert calculus}, European J. Combin. \textbf{25} (2004), no.~8,
  1263--1283.

\bibitem[Har77]{Har}
Robin Hartshorne, \emph{Algebraic geometry}. Graduate Texts in
  Mathematics vol 52, Springer-Verlag, 1977.

\bibitem[Jan87]{Jan}
Jens Carsten Jantzen, \emph{Representations of algebraic groups}, Pure 
  and Applied Mathematics vol.~131, Academic Press, Boston, 1987.

\bibitem[Kle74]{Kle}
Stephen Kleiman, \emph{The transversality of a generic translate},
  Compositio Math. \textbf{28} (1974), 287--297.

\bibitem[Mag98]{Mag}
Peter Magyar, \emph{Schubert polynomials and Bott-Samelson varieties}, 
  Comment. Math. Helv. \textbf{73} (1998), no.~4, 603--636.

\bibitem[Mat00]{Mat}
Olivier Mathieu, \emph{Positivity of some intersections in $K\sb 0(G/B)$}, 
  J. Pure Appl. Algebra \textbf{152} (2000), 231--243.

\bibitem[Mat89]{Matsumura}
Hideyuki Matsumura, \emph{Commutative ring theory}. Cambridge Studies in 
   Advanced Mathematics vol~8, Cambridge University Press, 1989.

\bibitem[Mih06]{mihalcea}
Leonardo Mihalcea, \emph{Positivity in equivariant quantum Schubert
  calculus}, Amer. J. Math. \textbf{128} (2006), no.~3, 787--803.

\bibitem[MiSp08]{torVanish}
Ezra Miller and David Speyer, \emph{A Kleiman--Bertini theorem for
  sheaf tensor products}, Journal of Algebraic Geometry \textbf{17}
  (2008), 335--340.  DOI: S 1056-3911(07)00479-1.
  \textsf{arXiv:math.AG/0601202}

\bibitem[Mum99]{Mum}
David Mumford, \emph{The red book of varieties and schemes}.  Second,
  expanded, edition.   Includes the Michigan lectures (1974) on curves
  and their Jacobians.  With contributions by Enrico Arbarello.
  Lecture Notes in Math.\ vol. 1358.  Springer-Verlag, 1999.

\bibitem[PiRa99]{pittie-ram}
Harsh Pittie and Arun Ram, \emph{A Pieri--Chevalley formula in
  the $K$-theory of a $G/B$-bundle}, Electron. Res. Announc. Amer.
   Math. Soc. \textbf{5} (1999), 102--107.

\bibitem[Ram85]{ramanathan}
Annamalai Ramanathan, \emph{Schubert varieties are arithmetically
  Cohen--Macaulay}, Invent. Math. \textbf{80} (1985), 283--294.

\bibitem[Ram87]{ramanathan2}
Annamalai Ramanathan, \emph{Equations defining Schubert varieties and 
  Frobenius splitting of diagonals}, Inst. Hautes \'Etudes Sci. Publ. Math. 
  {\bf 65} (1987), 61--90.

\bibitem[Sie07]{sierra}
Susan Sierra, \emph{A general homological Kleiman--Bertini theorem}.
  \textsf{arXiv:math.AG/0705.0055}

\bibitem[Spe88]{speiser}
Robert Speiser, \emph{Transversality theorems for families of maps},
  Algebraic geometry (Sundance, UT, 1986), 235--252, Lecture Notes in
  Math.\ vol. 1311, Springer-Verlag, 1988.

\bibitem[Tot99]{totaro}
Burt Totaro, \emph{The Chow ring of a classifying space}, Algebraic
  $K$-theory (Seattle, WA, 1997), American Mathematical Society,
  Providence, RI, 1999, pp.~249--281.

\bibitem[Vil92]{Vil92}
Orlando E. Villamayor U., \emph{Patching local uniformizations}, Ann.\
Sci.\ \'Ecole Norm.\ Sup.\ (4) \textbf{25} (1992), no.~6, 629--677.


\end{thebibliography}
\end{document}